\documentclass[reqno,a4paper,11pt]{article}
\usepackage{amsmath,amssymb,amsthm,amsfonts,mathrsfs}
\usepackage{color}

\theoremstyle{plain}
\newtheorem{theorem}{Theorem} [section]

\newtheorem{lemma}[theorem]{Lemma}
\newtheorem{proposition}[theorem]{Proposition}

\theoremstyle{definition}
\newtheorem{definition}[theorem]{Definition}

\newtheorem{remark}[theorem]{Remark}
\numberwithin{theorem}{section}
\numberwithin{equation}{section}
\numberwithin{figure}{section}

\def\N{\mathbb N}

\def\R{\mathbb R}

\def\a{\alpha}
\def\b{\beta}

\def\d{\delta}
\def\e{\varepsilon}
\def\eps{\varepsilon}

\def\l{\lambda}

\def\var{\varphi}
\def\Om{\Omega}

\newcommand{\A}{\mathcal A}
\newcommand{\B}{\mathcal B}
\newcommand{\C}{\mathcal C}
\renewcommand{\H}{\mathcal H}
\newcommand{\K}{\mathcal K}

\renewcommand{\S}{\mathcal S}

\DeclareMathOperator{\dist}{dist}
\DeclareMathOperator{\co}{co}

\title{Quantitative stability 
for the Brunn-Minkowski inequality}

\author{Alessio Figalli\thanks{The University of Texas at Austin,
Mathematics Dept. RLM 8.100,
2515 Speedway Stop C1200,
Austin, TX 78712-1202 USA.\, \textit{E-mail address:} \texttt{figalli@math.utexas.edu}} \,\,and David Jerison\thanks{Department of Mathematics, 
Massachusetts Institute of Technology, 
77 Massachusetts Ave,
Cambridge, MA 02139-4307
USA. \, \textit{E-mail address:} \texttt{jerison@math.mit.edu}}}
\date{}

\begin{document}
\maketitle

\begin{abstract}
We prove a quantitative stability result for the 
Brunn-Minkowski inequality: if $|A|=|B|=1$, $t \in [\tau,1-\tau]$ with $\tau>0$, and
$|tA+(1-t)B|^{1/n}\leq 1+\delta$ for some small $\delta$,
then, up to a translation, both $A$ and $B$ are quantitatively close (in terms of $\delta$)
to a convex set $\K$.
\end{abstract}

\section{Introduction}

Given two sets $A,B\subset \R^n$, and  $c>0$, we define the 
set sum and scalar multiple by 
\begin{equation}
\label{eq:def sum}
A + B := \{a + b: a\in A, \ b\in B\}, \quad cA := \{ ca: a\in A\} 
\end{equation}
Let $|E|$ denote the Lebesgue measure of  a set $E$ (if $E$ is not measurable,
$|E|$ denotes the outer Lebesgue measure of $E$).
The Brunn-Minkowski inequality states that, given $A,B\subset \R^n$ nonempty measurable sets,
\begin{equation}
\label{eq:BM}
|A+B|^{1/n}\geq |A|^{1/n}+|B|^{1/n}.
\end{equation}
In addition, if $|A|,|B|>0$, then equality holds if and only if there exist
a convex set $\K\subset \R^n$, $\l_1,\l_1>0$, and $v_1,v_2 \in \R^n$, such that
$$
\l_1A+ v_1 \subset  \K ,\quad \l_2 B+ v_2 \subset  \K,
\qquad |\K \setminus ( \l_1A + v_1)| 
=  |\K \setminus (\l_2 B + v_2)|  = 0.
$$
Our aim is to investigate the stability of such a statement.\\

When $n=1$, the following sharp stability result holds as a consequence of
classical theorems in additive combinatorics (an elementary proof of this result
can be given using Kemperman's theorem \cite{christ3,christ4}):
\begin{theorem}
\label{thm:n1}
Let $A,B\subset \R$ be measurable sets.
If $|A+B|< |A|+|B|+\delta$ for some $\delta \leq \min\{|A|,|B|\}$,
then there exist two intervals $I,J\subset \R$ such that $A\subset I$, $B\subset J$, $|I\setminus A| \leq \delta$, and $|J\setminus B| \leq \delta$. 
\end{theorem}
Concerning  the higher dimensional case, in
\cite{christ1,christ2} M. Christ proved a \textit{qualitative} stability result 
for \eqref{eq:BM}, namely, if
$|A+B|^{1/n}$ is close to $|A|^{1/n}+|B|^{1/n}$ then $A$ and $B$
are close to homothetic convex sets.  

On the \textit{quantitative} side,
first V. I. Diskant \cite{Disk} and then H. Groemer \cite{Groe}
obtained some stability results for {\em convex sets} in terms of the Hausdorff distance.
More recently, sharp stability results in terms of the $L^1$ distance have been obtained by 
the first author together with F. Maggi and A. Pratelli \cite{fmpK,fmpBM}. Since this latter result
will play a role in our proofs, we state it in detail.

We begin by noticing that, after dilating $A$ and $B$ appropriately,  we can assume $|A|=|B|=1$ while replacing the sum $A+B$ by a convex combination $S:=tA+(1-t)B$.
It follows by \eqref{eq:BM} that $|S|= 1+\delta$ for some $\delta \geq 0$.
\begin{theorem}\label{thm:BMconvex} (see \cite{fmpK,fmpBM})
There is a computable dimensional
constant $C_0(n)$ such that if $A,B\subset \R^n$ are convex sets satisfying $|A| = |B| = 1$, $|tA + (1-t)B| = 1 +\delta$ for some $t\in [\tau, 1-\tau]$, then, up to a translation,
\[
| A \Delta B| \le C_0(n) \tau^{-1/2n} \delta^{1/2}
\]
\end{theorem}
\noindent
(Here and in the sequel, $E\Delta F$ denotes the symmetric difference between two sets $E$ and $F$,
that is $E\Delta F=(E\setminus F)\cup(F\setminus E)$.)\\

Our main theorem here is a quantitative version of Christ's result.   His
result relies on compactness and, for that reason, does not yield any 
explicit information about the dependence on the parameter $\delta$.   
Since our proof is by induction on the dimension, it will be convenient to 
allow the measures of $|A|$ and $|B|$ not to be exactly equal, but just close in 
terms of $\delta$.  Here is the main result of this paper, which shows that the measure of the difference between the sets $A$ and $B$ 
and their convex hull is bounded by a power $\delta^\epsilon$, confirming
a conjecture of Christ \cite{christ1}.   

\begin{theorem}
\label{thm:main}
Let $n \geq 2$, let $A,B\subset \R^n$ be measurable sets, and define 
$S:=tA+(1-t)B$ for some $t \in [\tau,1-\tau]$, $0 < \tau \le 1/2$.
There are computable dimensional constants $N_n$ and 
computable functions $M_n(\tau),\e_n(\tau)>0$ such that if
\begin{equation}
\label{eq:measures}
\bigl||A|-1\bigr|+\bigl||B|-1\bigr|+\bigl||S|-1\bigr|\leq \d
\end{equation}
for some $\d\leq e^{-M_n(\tau)}$, then there exists a convex set 
$\K\subset \R^n$ such that, up to a translation, 
\[
A,B\subset \K\qquad \text{and}\qquad 
|\K\setminus A|+|\K\setminus B|\leq \tau^{-N_n} \d^{\e_n(\tau)}.
\]
Explicitly, we may take
$$
M_{n}(\tau) = \frac{2^{3^{n+2}} n^{3^n} |\log \tau|^{3^n}}{\tau^{3^n}},
\qquad
\e_{n}(\tau) = \frac{\tau^{3^n}}{2^{3^{n+1}} n^{3^n} |\log \tau|^{3^n}}.
$$
\end{theorem}
\noindent

It is interesting to make some comments on the above theorem:
first of all, notice that the result holds only under the assumption that $\delta$ is sufficiently small,
namely $\d\leq e^{-M_n(\tau)}$. A smallness assumption on $\delta$ is actually necessary, as can be easily seen from the following
example:
$$
A=B:=B_\rho(0)\cup \{2L \,e_1\},
$$
where $L \gg 1$,  $e_1$ denotes the first vector of the canonical basis in $\R^n$, and $\rho>0$ is chosen so that $|B_\rho(0)|=1$.
Then it is easily checked that
$$
\left|\textstyle{\frac12}A+\textstyle{\frac12}B\right|=\left| B_\rho(0)\cup B_{\rho/2}(L \,e_1) \cup \{2L \,e_1\}\right|=1+2^{-n},
$$
while
$|\co(A)| \approx L$ can be arbitrarily large, hence the result is false unless we assume that $\delta<2^{-n}$.

Concerning the exponent $\eps_n(\tau)$, at the moment it is unclear to us whether a dimensional dependency is necessary. It is however worth to point out that there are stability results for functional inequalities
where a dimensional dependent exponent is needed (see for instance \cite[Theorem 3.5]{BP}),
so it would not be completely surprising if in this situation the optimal exponent does depend on $n$.
We plan to investigate this very delicate question in future works.

Another important direction to develop would be to understand the analytic counterpart of the Brunn-Minkowski
inequality, namely the Pr\'ekopa-Leindler inequality. At the moment, some stability estimates are known only in one dimension
or for some special class of functions \cite{BB1,BB2},
and a general stability result would be an important direction of future investigations.
\\

The paper is structured as follows.  In the next section we introduce a few notations
and give an outline of the proof along with some commentary on 
the techniques and ideas.  Then, in Section \ref{sect:tech results} we collect 
most of the technical results we will use.  Since the proofs of some of these
technical results are delicate and involved, we postpone them to Section 
\ref{sect:proofs}.  Section \ref{sec:proof} is 
devoted to the proof of Theorem \ref{thm:main}.\\

\textit{Acknowledgements:}
AF was partially supported by NSF Grant DMS-1262411.
DJ was partially supported by the Bergman Trust and NSF Grant DMS-1069225.

\section{Notation and an outline of the proof}
\label{sect:outline}

Let $\H^k$ denote $k$-dimensional Hausdorff measure on $\R^n$.  
Denote by $x=(y,s) \in \R^{n-1}\times \R$ a point in $\R^n$,
and let $\pi:\R^n \to \R^{n-1}$ and $\bar \pi:\R^n\to \R$ denote the canonical projections, i.e.,
$$
\pi(y,s):= y\qquad\text{and}\qquad \bar\pi(y,s):=s.
$$
Given a compact set $E\subset \R^n$, $y \in \R^{n-1}$, and $\lambda>0$, we use the notation 
\begin{equation}
\label{eq:Ey}
E_y := E\cap \pi^{-1}(y)\subset \{y\}\times\R,\qquad E(s):=E\cap \bar\pi^{-1}(s)\subset \R^{n-1}\times\{s\},
\end{equation}
\begin{equation}
\label{eq:Elambda}
\mathcal E(\lambda):=\bigl\{y \in \R^{n-1}:\H^1(E_y)>\lambda\bigr\}.
\end{equation}
Following Christ \cite{christ2}, we consider different symmetrizations.

\begin{definition}
\label{def:symm}
Let $E\subset \R^n$ be a compact set.
We define the \textit{Schwarz symmetrization} $E^\ast$ of $E$ as follows. 
For each $t \in \R$,
\begin{enumerate}
\item[-] If $\H^{n-1}\bigl(E(s)\bigr)>0$, then
$E^\ast(s)$ is the closed disk centered at $0 \in \R^{n-1}$
with the same measure.
\item[-] If $\H^{n-1}\bigl(E(s)\bigr)=0$, then $E^\ast(s)$ is empty.
\end{enumerate}
We define the \textit{Steiner symmetrization} $E^\star$ of $E$ 
so that for each $y \in \R^{n-1}$, the set $E^\star_y$ is empty if $\H^1(E_y)=0$;
otherwise it is the closed interval of length $\H^1(E_y)$ centered at $0\in \R$.
Finally, we define $E^\natural:=(E^\star)^\ast$.
\end{definition}

\subsection*{Outline of the proof of Theorem \ref{thm:main}}

The proof of Theorem \ref{thm:main} is very elaborate, combining 
the techniques of M. Christ with those developed by the present authors 
in \cite{fjA} (where we proved Theorem \ref{thm:main} in the special case $A=B$ 
and $t=1/2$), as well as several new ideas.  For that reason, we give
detailed description of the argument. \\

In Section \ref{sect:natural} we prove the theorem
in the special case $A=A^\natural$ and $B=B^\natural$.
In this case we have that
$$
A_y =\{y\}\times [-a(y),a(y)] \quad \text{and}\quad B_y =\{y\}\times[-b(y),b(y)],
$$
for some functions $a,b:\R^{n-1}\to \R^+$,
and it is easy to show that $a$ and $b$ satisfy the ``3-point concavity inequality''
\begin{equation}
\label{eq:3pts}
ta(y')+(1-t)b(y'') \leq [ta+(1-t)b](y)+\delta^{1/4}
\end{equation}
whenever $y'$, $y''$, and $y:=ty'+(1-t)y''$ belong to a large subset 
$F$ of $\pi(A)\cap \pi(B)$.
From this $3$-point inequality and an elementary argument (Remark \ref{rmk:4pts})
we show that $a$ satisfies the ``4-point concavity inequality'' 
\begin{equation}
\label{eq:4pts ineq}
a(y_1)+a(y_2)\leq a(y_{12}')+a(y_{12}'')+\frac2t \delta^{1/4}
\end{equation}
with $y_{12}':=t'y_1+(1-t')y_2$, $y_{12}'':=t''y_1+(1-t'')y_2$, 
$t':=\frac{1}{2-t}$, $t'':=1-t'$, provided
all four points belong to $F$.  (The analogous inequality
for $b$ involves a different set of four points.)

Using this inequality and Lemma \ref{lemma:concavity}, we deduce
that $a$ is quantitatively close in $L^1$ to a concave function. 
The proof, in Section \ref{sect:proofs}, of Lemma \ref{lemma:concavity},
although reminiscent of Step 4 in the proof of \cite[Theorem 1.2]{fjA}, 
is delicate and involved. 

Once we know that $a$ (and analogously $b$) is $L^1$-close 
to a concave function,
we deduce that both $A$ and $B$ are $L^1$-close to convex sets $K_A$ and $K_B$
respectively, and
we would like to say that these convex sets are nearly the same.
This is demonstrated as part of Proposition \ref{prop:coS}, which is proved by
first showing that $S$ is close to $tK_A+(1-t)K_B$,
then applying Theorem \ref{thm:BMconvex} to deduce that
$K_A$ and $K_B$ are almost homothetic, and then constructing
a convex set $\K$ close to $A$ and $B$ and containing both of them.

This concludes the proof of Theorem 
\ref{thm:main} in the case $A=A^\natural$ and $B=B^\natural$.\\

In Section \ref{sect:general proof}
we consider the general case, which we prove in several steps, culminating
in induction on dimension. \\

{\it Step 1.} This first step is very close to the argument used by M. Christ in \cite{christ2},
although our analysis is more elaborate since we have to quantify every estimate.

Given $A$, $B$, and $S$, as in the theorem, we consider 
their symmetrizations 
$A^\natural$, $B^\natural$, and $S^\natural$,
and apply the result from Section \ref{sect:natural} to deduce that $A^\natural$ and $B^\natural$
are close to the same convex set.
This information combined with Christ's Lemma \ref{lem:symm} allows us to deduce that functions $y\mapsto \H^1(A_y)$ and $y\mapsto \H^1(B_y)$ are almost equipartitioned (that is, the measure
of their level sets $\A(\lambda)$ and $\B(\lambda)$ are very close). This fact combined with a Fubini argument
yields  that, for most levels $\lambda$,  $\A(\lambda)$ and $\B(\lambda)$ 
are almost optimal for the $(n-1)$-dimensional Brunn-Minkowski inequality.
Thus, by the inductive step, we can find a level $\bar\lambda \sim \d^\zeta$ ($\zeta>0$) such that we can apply the inductive hypothesis to $\A(\bar\lambda)$ and $\B(\bar\lambda)$.  Consequently, after removing sets of 
small measure both from $A$ and $B$ and translating in $y$, we deduce that 
$\pi(A),\pi(B)\subset \R^{n-1}$ are close to the same convex set.\\

{\it Step 2.} This step is elementary: we apply a Fubini argument and
Theorem \ref{thm:n1} to most of the sets $A_y$ and $B_y$ for $y \in
\A(\bar\lambda)\cap \B(\bar\lambda)$ to deduce
that they are close to their convex hulls. Note, however, that
to apply Fubini and Theorem \ref{thm:n1} it is crucial that, thanks to Step 1, 
we found a set in $\R^{n-1}$ onto which both $A$ and $B$ project almost fully.  
Indeed, in order to say
that $\H^1(A_y+B_y)\geq \H^1(A_y) +\H^1(B_y)$ it is necessary to know
that both $A_y$ and $B_y$ are \emph{nonempty}, as otherwise the
inequality would be false!\\

{\it Step 3.} The argument here uses several ideas from our previous paper \cite{fjA}
to obtain a  3-point concavity inequality as in \eqref{eq:3pts} above 
for the ``upper profile'' of $A$ and $B$ 
(and an analogous inequality for the ``lower profile'').
This inequality allows us to say that the barycenter of $A_y$ 
satisfies the 4-point inequality \eqref{eq:4pts ineq} both from above and 
from below, and from this information we can deduce that,
as a function of $y$, the barycenter of $A_y$ (resp. $B_y$) 
is at bounded distance from a linear function (see Lemma \ref{lemma:1d}).
It follows that the barycenters of $\bar S_y$ are a bounded distance
from a linear function for a set $\bar S$ which is almost of full
measure inside $S$.  Then a variation of \cite[Proof of Theorem 1.2, Step 3]{fjA}
allows us to show that, after an affine measure preserving transformation, 
$\bar S$ is universally bounded, that is, bounded in diameter by 
a constant of the form $C_n\tau^{-M_n}$ where $C_n$ and $M_n$ are dimensional
constants. 
\\

{\it Step 4.} By a relatively easy argument we find sets $A^\sim$ and $B^\sim$
of the form
\[
A^\sim = \bigcup_{y\in F} \{y\}\times [a^A(y), b^A(y)]
\qquad
B^\sim = \bigcup_{y\in F} \{y\}\times [a^B(y), b^B(y)]
\]
which are close to $A$ and $B$, respectively, and are universally bounded.\\

{\it Step 5.} This is a crucial step: we want to show that $A^\sim$ and $B^\sim$ 
are close to convex sets.
As in the case $A=A^\natural$ and $B=B^\natural$,
we would like to apply Lemma \ref{lemma:concavity} to deduce that $b^A$ and $b^B$
(resp. $a^A$ and $a^B$) are $L^1$-close to concave (resp. convex) functions.

The main issue is that the hypothesis of the lemma, in addition to asking
for boundedness and  concavity of $b^A$ and $b^B$ at most points, also 
requires that the level sets of $b^A$ and $b^B$ be close to their convex hulls.
To deduce this  we wish to show that most slices of $A^\sim$ and $B^\sim$
are nearly optimal in the Brunn-Minkowski inequality in dimension
$n-1$ and invoke the inductive hypothesis.  We achieve
this by an inductive proof of the Brunn-Minkowski 
inequality, based on combining
the validity of Brunn-Minkowski in dimension $n-1$ with $1$-dimensional 
optimal transport (see Lemma \ref{lemma:BM n n-1}).

An examination of this proof of the Brunn-Minkowski inequality
in the situation near equality shows that if $A$ and $B$ are 
almost optimal for the Brunn-Minkowski inequality in dimension $n$,
then for most levels $s$, the slices $A(s)$ and $B(T(s))$ have
comparable $(n-1)$-measure, where $T$ is the 1-dimensional optimal transport
map, and this pair of sets is almost optimal for 
the Brunn-Minkowski inequality in dimension $n-1$.  In particular,
we can apply the inductive hypothesis to deduce that most $(n-1)$-dimensional
slices are close to their convex hulls.

This nearly suffices to apply Lemma \ref{lemma:concavity}.  But
this lemma asks for control of the superlevel sets of the function $b^A$,
which a priori may be very different from the slices of $A^\sim$.
To avoid this issue, we simply replace $A^\sim$ and $B^\sim$
by auxiliary sets $A^-$ and $B^-$ which consist of
the top profile of $A^\sim$ and $B^\sim$ with a flat bottom,
so that the slices coincide with the superlevel sets of $b^A$ and $b^B$.
We then show that Lemma \ref{lemma:BM n n-1} applies to $A^-$ and $B^-$.
In this way, we end up proving that $A^\sim$ and $B^\sim$ are 
close to convex sets, as desired.\\

{\it Step 6.} Since $A^\sim$ and $B^\sim$ are close to $A$ and $B$ respectively,
we simply apply Proposition \ref{prop:coS} as before in \ref{sect:natural} 
to conclude the proof of the theorem.\\

{\it Step 7.} Tracking down the exponents in the proof,
we provide an explicit lower (resp. upper) bound on $\e_n(\tau)$
(resp. $M_n(\tau)$).


\section{Technical Results}
\label{sect:tech results}

In this section we state most of the important lemmas we will need.  The first
three are due to M. Christ (or are easy corollaries of his results).

It is well-known that both the Schwarz and the Steiner symmetrization preserve the 
measure of sets, while they decrease the measure of the semi-sum (see for instance
\cite[Lemma 2.1]{christ2}).  Also, as shown in \cite[Lemma 2.2]{christ2}, the 
$\natural$-symmetrization preserves the measure of the sets $\mathcal E(\lambda)$.
We combine these results into one lemma, and refer to
\cite[Section 2]{christ2} for a proof.
\begin{lemma}
\label{lem:symm}
Let $A,B\subset \R^n$ be compact sets. Then $|A|=|A^\ast|=|A^\star|=|A^\natural|$,
$$
|tA^\ast+(1-t)B^\ast| \leq |tA+(1-t)B|,\quad |tA^\star+(1-t)B^\star| \leq |tA+(1-t)B|,
\quad |tA^\natural+(1-t)B^\natural| \leq |tA+(1-t)B|,
$$
and, with the notation in \eqref{eq:Elambda},
$$
\bigl|A\setminus \pi^{-1}\bigl(\A(\lambda)\bigr)\bigr|
= \bigl|A^\natural\setminus \pi^{-1}\bigl(\A^\natural(\lambda)\bigr) \bigr|\qquad \text{and}\qquad \H^{n-1}\bigl(\A(\l)\bigr)=\H^{n-1}\bigl(\A^\natural(\l)\bigr)
$$
 for almost every $\lambda>0$.
\end{lemma}

Another important fact is that a bound on the measure of $tA+(1-t)B$ in terms 
of the measures of $A$ and $B$ implies bounds relating the sizes of
\[
\sup_{y}\H^1(A_y), \qquad \sup_{y}\H^1(B_y), \qquad \H^{n-1}\bigl(\pi(A)\bigr),
\qquad \H^{n-1}\bigl(\pi(B)\bigr).
\]
\begin{lemma}
\label{lem:rel meas}
Let $A,B\subset \R^n$ be compact sets such that $|A|,|B| \geq 1/2$ and $|tA+(1-t)B|\leq 2$ for some $t \in (0,1)$,
and set $\tau:=\min\{t,1-t\}$
There exists a dimensional constant $M>1$ such that
$$
\frac{\sup_{y}\H^1(A_y)}{\sup_{y}\H^1(B_y)} \in \biggl(\frac{\tau^{n}}M,\frac{M}{\tau^{n}}\biggr),
\qquad \frac{\H^{n-1}\bigl(\pi(A)\bigr)}{\H^{n-1}\bigl(\pi(B)\bigr)} \in \biggl(\frac{\tau^{n}}M,\frac{M}{\tau^{n}}\biggr),
$$
$$
\Bigl( \sup_{y}\H^1(A_y)\Bigr)\H^{n-1}\bigl(\pi(A)\bigr)\in \biggl(\frac{1}M,\frac{M}{\tau^{2n}}\biggr),\qquad
\Bigl(\sup_{y}\H^1(B_y)\Bigr)\H^{n-1}\bigl(\pi(B)\bigr)\in \biggl(\frac{1}M,\frac{M}{\tau^{2n}}\biggr).
$$
and, up a measure preserving affine transformation of the form $(y,s)\mapsto (\lambda y,\lambda^{1-n}t)$ with $\lambda>0$, we have
\begin{equation}
\label{eq:normalized}
\H^{n-1}\bigl(\pi(A)\bigr)+\H^{n-1}\bigl(\pi(B)\bigr)+\sup_{y}\H^1(A_y)+\sup_{y}\H^1(B_y)\leq \frac{M}{\tau^{2n}}.
\end{equation}
In this case, we say that $A$ and $B$ are ($M,\tau$)-normalized.
\end{lemma}

\begin{proof}
As observed in \cite[Lemma 3.1]{christ2} and in the discussion immediately after that lemma,
$$
\Bigl(\sup_{y}\H^1(A_y)\Bigr)\H^{n-1}\bigl(\pi(B)\bigr) \leq \frac{|tA+(1-t)B|}{t(1-t)^{n-1}}\leq \frac{2}{\tau^n},\qquad
\Bigl(\sup_{y}\H^1(A_y)\Bigr)\H^{n-1}\bigl(\pi(A)\bigr)
\geq |A|\geq 1/2.
$$
By exchanging the roles of $A$ and $B$, the first part of the lemma follows.
To prove the second part, it suffices to choose $\lambda>0$ so that $\lambda^{n-1}\H^{n-1}\bigl(\pi(A)\bigr)=1/\tau^n$.
\end{proof}

The third lemma is a result of Christ \cite[Lemma 4.1]{christ1}
showing that $\sup_{s}\H^{n-1}\bigl(A(s)\bigr)$ and $\sup_{s}\H^{n-1}\bigl(B(s)\bigr)$
are close in terms of $\d$:
\begin{lemma}
\label{lem:pi close}
Let $A,B\subset \R^n$ be compact sets, define $S:=tA+(1-t)B$ for some $t \in [\tau,1-\tau]$,
and assume that \eqref{eq:measures} holds for some $\delta \leq 1/2$.
Then there exists a numerical constant $L>0$ such that
$$
\frac{\sup_{s}\H^{n-1}\bigl(A(s)\bigr)}{\sup_{s}\H^{n-1}\bigl(B(s)\bigr)} \in \bigl(1-L\tau^{-1/2}\d^{1/2},1+L\tau^{-1/2}\d^{1/2}\bigr).
$$
\end{lemma}
\begin{proof}
Set 
$$
\gamma:=\biggl(\frac{\sup_{s}\H^{n-1}\bigl(A(s)\bigr)}{\sup_{s}\H^{n-1}\bigl(B(s)\bigr)}\biggr)^{1-t},
\qquad \tilde\gamma:=\biggl(\frac{\sup_{s}\H^{n-1}\bigl(B(s)\bigr)}{\sup_{s}\H^{n-1}\bigl(A(s)\bigr)}\biggr)^{t},
$$
and after possibly exchanging $A$ and $B$, we may assume that $\gamma \leq 1$.
By the argument in the proof of \cite[Lemma 4.1]{christ1} we get
$$
|S|\geq t\gamma^{-1}|A|+(1-t)\tilde\gamma^{-1}|B|,
$$
so, by \eqref{eq:measures},
$$
\Bigl(t\gamma^{-1}+ (1-t)\gamma^{t/(1-t)} \Bigr)-1
\leq 4\delta.
$$
The function
$$
\gamma\mapsto t\gamma^{-1}+ (1-t)\gamma^{t/(1-t)}
$$
is convex for $\gamma \in (0,1]$, attains its minimum  at $\gamma=1$,
and its second derivative is bounded below by $\tau$.  It follows
that 
$$
4 \delta \geq 
\Bigl(t\gamma^{-1}+ (1-t)\gamma^{t/(1-t)} \Bigr)-1 
\geq \frac{\tau}{2} |\gamma-1|^2,
$$
which proves the result.
\end{proof}

There are several other important ingredients in the proof of 
Theorem \ref{thm:main}, which are to our knowledge new.
Because their proofs are long and involved, we postpone them 
to Section \ref{sect:proofs}.

The first of these results shows that if $A$ and $B$ are $L^1$-close to convex sets $K_A$ and $K_B$
respectively, then $A$ and $B$ are close to each other, and
we can find a convex set $\K$ which contains both $A$ and $B$ with a good control on the measure.
As we shall see, the proof relies primarily on Theorem \ref{thm:BMconvex}.
\begin{proposition}
\label{prop:coS}
Let $A,B\subset \R^n$ be compact sets, define $S:=tA+(1-t)B$ for some $t \in [\tau,1-\tau]$,
and assume that \eqref{eq:measures} holds. Suppose $A,B\subset B_R$, 
for some $R \leq \tau^{-N_n}$ with $N_n$ a dimensional constant $N_n>1$.
Suppose further that we can find a convex sets $K_A,K_B
\subset \R^n$ such that 
\begin{equation}
\label{eq:KA KB zeta}
|A\Delta K_A|+|B\Delta K_B|\leq \zeta
\end{equation}
for some $\zeta \geq \delta$.
Then there exists a dimensional constant $L_n>1$ such that
after a translation,
$$
|A\Delta B| \leq \tau^{-L_n}\zeta^{1/2n}
$$
and there exists a convex set $\K$ containing both $A$ and $B$ such that
$$
|\K\setminus A|+|\K\setminus B| \leq \tau^{-L_n}\zeta^{1/2n^3}.
$$
\end{proposition}


Our next result is a consequence of a proof of the Brunn-Minkowski
inequality by induction, using horizontal $(n-1)$-dimensional slices. 
The lemma says that when $A$, $B$, and $S$ satisfy \eqref{eq:measures}, then
most of their horizontal slices  (chosen at suitable levels) satisfy near 
equality in the Brunn-Minkowski inequality and the ratio of their volumes 
of the slices is comparable to $1$
on a large set. 
\begin{lemma} \label{lemma:BM n n-1}
Given compact sets $A,B\subset\R^n$ and $S:=tA+(1-t)B$, and recalling the notation 
$E(s)\subset \R^{n-1}\times\{s\}$ in \eqref{eq:Ey}, we define the probability 
densities on the real line
\begin{equation}
\label{eq:rhoABS}
\rho_A(s):=\frac{\H^{n-1}\bigl(A(s)\bigr)}{|A|},\qquad \rho_B(s):=\frac{\H^{n-1}\bigl(B(s)\bigr)}{|B|},
\qquad \rho_S(s):=\frac{\H^{n-1}\bigl(S(s)\bigr)}{|S|}.
\end{equation}
Let $T:\R\to \R$ be the monotone rearrangement sending $\rho_A$ onto $\rho_B$, that is $T$ is an increasing map such that $T_\sharp\rho_A=\rho_B$.\footnote{$T_\sharp$ 
denotes the push-forward through the map $T$, that is,
$$
T_\sharp\rho_A=\rho_B \qquad \Leftrightarrow \qquad \int_E\rho_B(s)\,ds=
\int_{T^{-1}(E)} \rho_A(s)\,ds\quad\forall\,\text{$E\subset \R$ Borel.}
$$ 
An explicit formula for $T$ can be given using the distribution functions of $\rho_A$ and $\rho_B$: if we define
$$
G_A(s):=\int_{-\infty}^s \rho_A(s')\,ds',
\qquad G_B(s):=\int_{-\infty}^s \rho_B(s')\,ds',
$$
and we set $G_B^{-1}(r):=\inf\{s \in \R \,:\,G_A(s)>t\}$, then 
$T=G_B^{-1}\circ G_A$. }
Then
\begin{equation}
\label{eq:BM n n-1}
|S|-\Bigl(t|A|^{1/n}+(1-t)|B|^{1/n}\Bigr)^n\geq
\int_\R e_{n-1}(s) \bigl(t + (1-t)T'(s)\bigr)\, ds,
\end{equation}
where $T_t(s) := ts + (1-t)T(s)$ and
\[
e_{n-1}(s) : = \H^{n-1}\bigl(S(T_t(s))\bigr)-
\left[t\H^{n-1}\bigl(A(s)\bigr)^{1/(n-1)}+(1-t)\H^{n-1}\bigl(B(T(s))\bigr)^{1/(n-1)} \right]^{n-1}.
\]
Moreover,  if  $t \in [\tau,1-\tau]$ and \eqref{eq:measures} holds
with $\delta/\tau^n$ is sufficiently small, then 
\begin{equation}
\label{induction:key}
\int_\R \biggl| \frac{\rho_A(s)}{\rho_B(T(s))} - 1\biggr|\,\rho_A(s)\,ds \leq \frac{C(n)}{\tau^{n/2}} \delta^{1/2}.
\end{equation}
\end{lemma}

Finally, we have a lemma saying that if a function $\psi$ is nearly
concave on a large set, and most of its level sets are close to their
convex hulls, then it is $L^1$-close to a concave function.
Here and in the sequel, given a set $E$ we will use $\co(E)$
to denote its convex hull.
\begin{lemma}
\label{lemma:concavity}  Let $0 < \tau \le1/2$ and fix $t'$ such that 
$1/2 \le t' \le 1-\tau/2$.  Let $t'' = 1-t'$, and for all 
$y_1$ and $y_2$ in $\R^{n-1}$ define
\[
y_{12}' := t'y_1 + t'' y_2; \quad y_{12}'' := t'' y_1 + t' y_2.
\]
Let $\sigma,\varsigma>0$, $\hat M \geq 1$, $F\subset \R^{n-1}$, and
let $\psi:F\to \R$ be a function satisfying
\begin{equation}
\label{eq:ab}
\psi(y_1)+\psi(y_2)\leq \psi(y_{12}')+\psi(y_{12}'')+ \sigma \qquad \forall\,y_1,y_2,y_{12}',y_{12}'' \in F,
\end{equation}
\begin{equation}
\label{eq:coF}
\Omega:=\co(F),\qquad
 \H^{n-1}(\Omega\setminus F) \leq
\varsigma,
\end{equation}
\begin{equation}
\label{eq:Ks ball 2}
B_r\subset \Omega \subset B_{(n-1)r},\qquad 1/n <r< n,
\end{equation}
\begin{equation}
\label{eq:abM}
-\hat M\leq \psi(y) \leq \hat M\qquad \forall\,y \in F.
\end{equation}
Also, we assume that there exists a set $H\subset \R$ such that
\begin{equation}
\label{eq:level sets co}
\int_{H} \H^{n-1}\bigl(\co(\{\psi>s\})\setminus \{\psi>s\}\bigr)\,ds + \int_{\R\setminus H}\H^{n-1}\bigl(\{\psi>s\}\bigr)\,ds\leq \varsigma.
\end{equation}
Then there exist a concave function $\Psi:\Omega\to [-2\hat M,2\hat M]$ and a dimensional constant $L_n'$ such that  
\begin{equation}
\label{eq:almost concave}
\int_{F} |\Psi(y)-\psi(y)|\,dy \leq \tau^{-L_n'}\,\hat M \,(\sigma+\varsigma)^{\beta_{n,\tau}},
\end{equation}
where 
\[
\beta_{n,\tau}:= \frac{\tau}{16(n-1)|\log \tau|}
\]
\end{lemma}

\section{Proof of Theorem \ref{thm:main}}\label{sec:proof}

As explained in \cite{fjA}, by inner approximation\footnote{The approximation
of $A$ (and analogously for $B$) is by a sequence of compact sets $A_k\subset A$ such that
$|A_k| \to |A|$ and $|\co(A_k)| \to |\co(A)|$. One way to construct such sets is to define $A_k := A_k' \cup V_k$, where $A_k'\subset A$ are compact sets
satisfying $|A_k'| \to |A|$, and $V_k\subset V_{k+1}\subset A$ are finite sets satisfying $|\co(V_k)| \to |\co(A)|$.} 
it suffices to prove the result when $A,B$
are compact sets. Hence, let  $A$ and $B$ be compact sets,
define $S:=tA+(1-t)B$ for some $t \in [\tau,1-\tau]$,
and assume that \eqref{eq:measures} holds.
We want to prove that there exists a convex set $\K$ such that, up to a translation,
$$
A,B\subset \K,\qquad |\K\setminus A|+|\K\setminus B|\leq 
\tau^{-N_n} \d^{\e_n(\tau)}.
$$
In order to simplify the notation, $C$ will denote a generic constant, 
which may change from line to line, and that is bounded from above by $\tau^{-N_n}$
for some dimensional constant $N_n>1$ (recall that by assumption $\tau\leq 1/2$).
We will say that such a constant is {\it universal}.

Observe that, since the statement and the conclusions are invariant under 
measure preserving affine transformations,
by Lemma \ref{lem:rel meas}  we can assume that $A$ and $B$ are 
($M,\tau$)-normalized (see \eqref{eq:normalized}).

\subsection{The case $A=A^\natural$ and $B=B^\natural$}
\label{sect:natural} 
Let $A,B\subset \R^n$ be compact sets satisfying $A=A^\natural$, $B=B^\natural$.
Since
$$
\pi\bigl(A(s)\bigr) \subset \pi\bigl(A(0)\bigr) = \pi(A) 
\quad \text{and} \quad \pi\bigl(B(s)\bigr) \subset \pi\bigl(B(0)\bigr) = \pi(B) 
\quad \text{are disks centered at the origin},
$$
applying Lemma \ref{lem:pi close} we deduce that
\begin{equation}
\label{eq:pi close}
\H^{n-1}\bigl(\pi(A)\Delta\pi(B)\bigr)\leq C\,\delta^{1/2}.
\end{equation}
Hence, if we define
$$
\bar S:=\bigcup_{y \in \pi(A)\cap \pi(B)}tA_y+(1-t)B_y,
$$
then $\bar S_y\subset S_y$ for all $y \in \R^{n-1}$.
In addition, using \eqref{eq:measures}, \eqref{eq:normalized}, and \eqref{eq:pi close}, 
we have
\begin{align*}
1+\delta
&\geq |S|=\int_{\R^{n-1}}\H^{1}(S_y)\,dy \geq \int_{\pi(A)\cap \pi(B)} \H^{1}(S_y)\,dy \geq \int_{\pi(A)\cap \pi(B)} \H^{1}(\bar S_y)\,dy\\
&=|\bar S| 
\geq 
 t\int_{\pi(A)\cap \pi(B)}  \H^{1}(A_y)\,dy
+(1-t)\int_{\pi(A)\cap \pi(B)}  \H^{1}(B_y)\,dy\\
&\geq \frac{t|A|+(1-t)|B|}{2} - C\, \H^{n-1}\bigl(\pi(A)\Delta\pi(B)\bigr)
\geq 1-C\,\delta^{1/2},
\end{align*}
which implies  (since $\bar S\subset S$)
\begin{equation}
\label{eq:SbarSnatural}
|S\setminus \bar S| \leq C\,\delta^{1/2}.
\end{equation}
Also, by Chebyshev's inequality we deduce that there exists a set $F\subset \pi(A)\cap\pi(B)$ such that
$$
\H^{n-1}\bigl((\pi(A)\cap\pi(B))\setminus F\bigr) \leq C\,\delta^{1/4},\qquad \H^{1}\bigl(S_y\setminus \bar S_y\bigr) \leq \delta^{1/4} \quad \forall\,y \in F.
$$
This implies that, if we write
$$
A_y:=\{y\}\times [-a(y),a(y)] \quad \text{and}\quad B_y:=\{y\}\times[-b(y),b(y)],
$$
with $a$ and $b$ radial decreasing, then
$$
ta(y')+(1-t)b(y'') \leq [ta+(1-t)b](y)+\delta^{1/4} \qquad \forall\,y=ty'+(1-t)y'',\,\, y,y',y'' \in F.
$$

We show next that a three-point inequality for two functions $f$ and $g$ 
implies a four-point inequality for each of $f$ and $g$ separately.\\
\begin{remark}
\label{rmk:4pts}
Let $F\subset \R^{n-1}$, and $f,g:F\to \R$ be two bounded Borel functions satisfying
$$
tf(y')+(1-t)g(y'') \leq [tf+(1-t)g](y)+\sigma \qquad \forall\,y=ty'+(1-t)y'',\,\, y,y',y'' \in F,
$$
for some $\sigma\geq 0$.  Let $t\in [\tau, 1-\tau]$ and define
\begin{equation}
\label{eq:t'}
t':=\frac{1}{2-t},\quad t'':= 1 -t';\qquad
y_{12}':=t'y_1+(1-t')y_2, \qquad y_{12}'':=t''y_1+(1-t'')y_2
\end{equation}
We claim that 
\begin{equation}
\label{eq:4pts}
f(y_1)+f(y_2)\leq f(y_{12}')+f(y_{12}'')+\frac2t \sigma.
\end{equation}
(The analogous statement for $g$ involves replacing $t$ with $1-t$,
so gives different values of $t'$ and $t''$.) 
Notice that, if $\tau \le t\le 1/2$, then 
\begin{equation}
\label{eq:tau12}
1/2 \le t' \le 2/3
\end{equation}
independent of $\tau$, whereas if $1/2 \le t \le 1-\tau$, then
\begin{equation}
\label{eq:tau12'}
2/3 \le t' \le 1 -\tau/2.
\end{equation}
\end{remark}
To prove \eqref{eq:4pts}, note that the definitions above imply
\[
y_{12}'=ty_1+(1-t)y_{12}'', \quad y_{12}''=ty_{12}'+(1-t)y_2.
\]
Hence, assuming that $y_1,y_2,y_{12}',y_{12}'' \in F$, we can add together
the two inequalities
$$
tf(y_1)+(1-t)g(y_{12}'') \leq [tf+(1-t)g](y_{12}')+\sigma,
$$
$$
tf(y_2)+(1-t)g(y_{12}') \leq [tf+(1-t)g](y_{12}'')+\sigma,
$$
to get \eqref{eq:4pts}.
\\

By the remark above and Lemma \ref{lemma:concavity} (notice that the level sets of $a$ and $b$ are both disks, so \eqref{eq:level sets co} holds with $\varsigma=0$),
we obtain that both functions $a$ and $b$ are $L^1$-close to concave functions $\Psi_A$ and $\Psi_B$, both defined on $\pi(A)\cap\pi(B)$.
Hence, if we define the convex sets
$$
K_A:=\bigl\{(y,s) \in \R^n\,:\,y\in \pi(A)\cap\pi(B),\ -\Psi_A(y) \leq s \leq \Psi_A(y)\bigr\},
$$
$$
K_B:=\bigl\{(y,s) \in \R^n\,:\,y\in \pi(A)\cap\pi(B),\ -\Psi_B(y) \leq s \leq \Psi_B(y)\bigr\},
$$
we deduce that
$$
|A\Delta K_A| +|B\Delta K_B|\le C\,\delta^{\beta_{n,\tau}/4}.
$$
Hence, it follows from Proposition \ref{prop:coS} that,
up to a translation, there exists a convex set $K$ such that $A\cup B\subset K$ and 
\begin{equation}
\label{eq:ABKnatural}
|A\Delta B|\leq C\,\delta^{\beta_{n,\tau}/8n}, \qquad |K\setminus A|+|K\setminus B|\leq C\,\delta^{\beta_{n,\tau}/8n^3}.
\end{equation}
Notice that, because $A=A^\natural$ and $B=B^\natural$, it is easy to check that the above properties still hold with $K^\natural$ in place of $K$. Hence, in this case, without loss of generality
one can assume that $K=K^\natural$.

\subsection{The general case}
\label{sect:general proof}
Since the result is true when $n=1$ (by Theorem \ref{thm:n1}), we assume that we already proved Theorem \ref{thm:main} through $n-1$, and we want to show its validity for $n$.

\subsubsection*{Step 1: There exist a dimensional constant $\zeta >0$ and $\bar\lambda \sim \d^\zeta$ such that the inductive hypothesis applies 
to $\A(\bar\lambda)$ and $\B(\bar \lambda)$.}
Let $A^\natural$ and $B^\natural$ be as in Definition \ref{def:symm}.
Thanks to Lemma \ref{lem:symm}, $A^\natural$
and $B^\natural$ still satisfy \eqref{eq:measures},
so we can apply the result proved in Section \ref{sect:natural} above to get (see \eqref{eq:ABKnatural})
\begin{equation}
\label{eq:ABsharp}
\int_{\R^{n-1}} \bigl|\H^1\bigl(A^\natural_y \bigr)
-\H^1\bigl(B^\natural_y \bigr)\bigr|\,dy \leq
\int_{\R^{n-1}} \bigl|\H^1\bigl(A^\natural_y \Delta B^\natural_y \bigr)\bigr|\,dy=|A^\natural\Delta B^\natural| \leq C\,\delta^{\bar\a}
\end{equation}
and 
\begin{equation}
\label{eq:ABsharpK}
K\supset A^\natural\cup B^\natural,\qquad |K\setminus A^\natural|+|K\setminus B^\natural| \leq C\,\delta^{\bar\a/n^2}
\end{equation}
for some convex set $K=K^\natural$,
where 
\begin{equation}
\label{eq:beta1}
\bar\a:=\frac{\beta_{n,\tau}}{8n}.
\end{equation}
In addition, because $A$ and $B$ are ($M,\tau$)-normalized (see \eqref{eq:normalized}), so are $A^\natural$ and $B^\natural$,
and by \eqref{eq:ABsharpK} we deduce that there exists a universal constant $R>0$ such that
\begin{equation}
\label{eq:KBR}
K\subset B_{R}.
\end{equation}
Also, by \eqref{eq:ABsharp} and Chebyshev's inequality we obtain that, up to a set of measure $\leq C\,\delta^{\bar\a/2}$, 
$$
\bigl|\H^1\bigl(A^\natural_y \bigr)
-\H^1\bigl(B^\natural_y \bigr)\bigr| \leq \delta^{\bar\a/2}.
$$
Thus, recalling Lemma \ref{lem:symm}, for almost every $\lambda>0$
$$
\H^{n-1}\bigl(\A(\lambda) \bigr)
=\H^{n-1}\bigl(\A^\natural(\lambda) \bigr) \leq
\H^{n-1}\bigl(\B^\natural(\lambda -\delta^{\bar\a/2} ) \bigr)+C\,\delta^{\bar\a/2}= \H^{n-1}\bigl(\B(\lambda -\delta^{\bar\a/2} ) \bigr)+C\,\delta^{\bar\a/2}.
$$
Since, by \eqref{eq:normalized},
$$
\int_0^{\tau^{-2n}M} \Bigl(\H^{n-1}\bigl(\B(\lambda) \bigr)- \H^{n-1}\bigl(\B(\lambda+\delta^{\bar\a/2}) \bigr)\Bigr)\,d\lambda=\int_0^{\delta^{\bar\a/2}}\H^{n-1}\bigl(\B(\lambda) \bigr) \,d\lambda 
\leq C\,\delta^{\bar\a/2},
$$
by Chebyshev's inequality we deduce that 
$$
\H^{n-1}\bigl(\A(\lambda) \bigr) \leq \H^{n-1}\bigl(\B(\lambda) \bigr)+C\,\delta^{\bar\a/4}
$$
for all $\lambda$ outside a set of measure $\delta^{\bar\a/4}$. 
Exchanging the roles of $A$ and $B$ we obtain that there exists a set $G\subset [0,\tau^{-2n}M]$
such that 
\begin{equation}
\label{eq:E}
\H^1(G)\leq C\,\delta^{\bar\a/4},\qquad \bigl| \H^{n-1}\bigl(\A(\lambda) \bigr) -\H^{n-1}\bigl(\B(\lambda) \bigr) \bigr| \leq C\,\delta^{\bar\a/4} \quad \forall\, \lambda\in [0,\infty]\setminus G.
\end{equation}
Using the elementary inequality
$$
\Bigl(ta+(1-t)b\Bigr)^{n-1}\geq ta^{n-1}+(1-t)b^{n-1} - C|a-b|^2 \qquad \forall\,0 \leq a,b \leq \frac{M}{\tau^{2n}},
$$
and replacing $a$ and $b$ with $a^{1/(n-1)}$ and $b^{1/(n-1)}$, respectively, we get
\begin{equation}
\label{eq:ab conc}
\Bigl(ta^{1/(n-1)}+(1-t)b^{1/(n-1)}\Bigr)^{n-1}\geq ta+(1-t)b - C|a-b|^{2/(n-1)} \qquad \forall\,0 \leq a,b \leq \frac{M}{\tau^{2n}}
\end{equation}
(notice that $|a^{1/(n-1)}-b^{1/(n-1)}| \leq |a-b|^{1/(n-1)}$).
Finally, it is easy to check that
$$
t\A(\l)+(1-t)\B(\l) \subset \S(\l)\qquad\forall\,\l>0.
$$
Hence, by the Brunn-Minkowski inequality \eqref{eq:BM} applied to $\A(\l)$ and $\B(\l)$,
using \eqref{eq:measures}, \eqref{eq:normalized}, \eqref{eq:ab conc}, and \eqref{eq:E}, we get
\begin{equation}
\label{eq:BMn-1}
\begin{split}
1+\delta \geq |S|&= \int_0^{\tau^{-2n}M} \H^{n-1}\bigl(\S(\lambda)\bigr) \,d\lambda\\
&\geq \int_0^{\tau^{-2n}M}  \Bigl(t\H^{n-1}\bigl(\A(\lambda)\bigr)^{1/(n-1)}
+(1-t)\H^{n-1}\bigl(\B(\lambda)\bigr)^{1/(n-1)}\Bigr)^{n-1} \,d\lambda\\
&\geq \int_0^{\tau^{-2n}M}  \Bigl(t\H^{n-1}\bigl(\A(\lambda)\bigr)
+(1-t)\H^{n-1}\bigl(\B(\lambda)\bigr)\Bigr) \,d\lambda\\
&\qquad - C \int_0^{\tau^{-2n}M} \bigl| \H^{n-1}\bigl(\A(\lambda) \bigr) -\H^{n-1}\bigl(\B(\lambda) \bigr) \bigr|^{2/(n-1)}\,d\lambda\\
&=t|A|+(1-t)|B| - C\,\delta^{\bar\a/[2(n-1)]}\\
&\geq 1-C\,\delta^{\bar\a/[2(n-1)]}.
\end{split}
\end{equation}
We also observe that, since $K=K^\natural$,
by Lemma \ref{lem:symm}, \eqref{eq:KBR},
and \cite[Lemma 4.3]{christ2},
for almost every $\lambda>0$ we have
\begin{equation}
\label{eq:AAlambda}
\begin{split}
\bigl|A\setminus \pi^{-1}\bigl(\A(\lambda)\bigr)\bigr|&
= \bigl|A^\natural\setminus \pi^{-1}\bigl(\A^\natural(\lambda)\bigr) \bigr|\\
&\leq \bigl|K\setminus \pi^{-1}\bigl(\K(\lambda)\bigr)\bigr|+ M\,\H^{n-1}\bigl(\A^\natural(\lambda)\Delta \K(\lambda)\bigr)\\
& \leq C \lambda^2 +M\,\H^{n-1}(\A^\natural(\lambda)\Delta \K(\lambda)),
\end{split}
\end{equation}
and analogously for $B$.
Also,  by \eqref{eq:ABsharpK},
\begin{equation}
\label{eq:AKnatural lambda}
\int_0^{\tau^{-2n}M} \Bigl(\H^{n-1}\bigl(\A^\natural(\lambda)\Delta \K(\lambda)\bigr)+\H^{n-1}\bigl(\B^\natural(\lambda)\Delta \K(\lambda)\bigr)\Bigr)\,d\lambda \leq |K\setminus A^\natural| +|K\setminus B^\natural| \leq C\,\delta^{\bar\a/n^2}.
\end{equation}
We set
\begin{equation}
\label{eq:eta}
\eta:=\frac{\bar\a}{n^2},
\end{equation}
and we notice that $\eta \leq 
\min\left\{\frac{\bar\a}{2(n-1)}, \frac{\bar\a}4\right\}$.
\\

Take $\zeta>0$ to be fixed later. Then 
by \eqref{eq:E}, \eqref{eq:BMn-1}, 
\eqref{eq:AAlambda}, \eqref{eq:AKnatural lambda}, and by Chebyshev's inequality, we can find a level
\begin{equation}
\label{eq:barlambda}
\bar \lambda \in \biggl[\frac{10\,\delta^\zeta}{\tau},\frac{20\,\delta^\zeta}{\tau}\biggr]
\end{equation}
such that
\begin{equation}
\label{eq:small deficit}
\H^{n-1}\bigl(\S(\bar \lambda)\bigr)\leq \Bigl(t\H^{n-1}\bigl(\A(\bar \lambda)\bigr)^{1/(n-1)}
+(1-t)\H^{n-1}\bigl(\B(\bar \lambda)\bigr)^{1/(n-1)}\Bigr)^{n-1} +C\,\delta^{\eta-\zeta},
\end{equation}
\begin{equation}
\label{eq:measure outside C}
\bigl|A\setminus \pi^{-1}\bigl(\A(\bar\lambda)\bigr)\bigr|
+\bigl|B\setminus \pi^{-1}\bigl(\B(\bar\lambda)\bigr)\bigr|
 \leq C\Bigl(\delta^{2\zeta} +\delta^{\eta-\zeta}\Bigr),
\end{equation}
\begin{equation}
\label{eq:close measure}
\bigl| \H^{n-1}\bigl(\A(\bar\lambda) \bigr)
-\H^{n-1}\bigl(\B(\bar\lambda) \bigr) \bigr| \leq C\,\delta^{\eta}.
\end{equation}
In addition, from the properties 
$$
\left\{
\begin{array}{l}
\text{$\H^{n-1}\bigl(\A(\lambda)\bigr)\leq \tau^{-2n}M$ for any $\lambda>0$ (see \eqref{eq:normalized})},\\
\text{$\int_0^{\tau^{-2n}M} \H^{n-1}\bigl(\A(\lambda)\bigr)\,d\l=|A|\geq 1-\d$},\\ 
\text{$\lambda \mapsto \H^{n-1}\bigl(\A(\lambda)\bigr)$ is a decreasing function,}
\end{array}
\right.
$$
we deduce that
$$
\frac{\tau^{2n}}{2M} \leq \H^{n-1}\bigl(\A(\lambda)\bigr)\leq  \frac{M}{\tau^{2n}} \qquad \forall\,\l\in \bigl(0,\tau^{2n}(2M)^{-1}\bigr).
$$
The same holds for $B$ and $S$, hence
$$
\H^{n-1}\bigl(\S(\bar \lambda)\bigr),\H^{n-1}\bigl(\A(\bar \lambda)\bigr),
\H^{n-1}\bigl(\B(\bar \lambda)\bigr) \in \bigl[\tau^{2n}(2M)^{-1},\tau^{-2n}M\bigr]
$$
provided $\delta\leq \tau^{-N}$ for some large dimensional constant $N$.
Set $\rho:=1/\H^{n-1}\bigl(\A(\bar\lambda)\bigr)^{1/(n-1)}\in [1/C,C]$, and define
$$
A':=\rho\A(\bar\lambda),\qquad
B':=\rho\B(\bar\lambda),\qquad
S':=\rho\S(\bar\lambda).
$$
By \eqref{eq:small deficit} and \eqref{eq:close measure} we get
$$
\H^{n-1}(A')=1,  \quad \bigl|\H^{n-1}(B')-1\bigr| \leq C\,\delta^{\eta},\quad \H^{n-1}(S') \leq 1+
C\,\delta^{\eta-\zeta}.
$$
while, by \eqref{eq:BM},
$$
\H^{n-1}(S')^{1/(n-1)}\geq t\H^{n-1}(A')^{1/(n-1)}+(1-t)\H^{n-1}(B')^{1/(n-1)}\geq 1-C\,\delta^{\eta},
$$
therefore
$$
\bigl|\H^{n-1}(A')-1\bigr|+\bigl|\H^{n-1}(B')-1\bigr|+\bigl|\H^{n-1}(S')-1\bigr|\leq C\,\delta^{\eta-\zeta}.
$$
Thus, by Theorem \ref{thm:main} applied with $n-1$, up to a translation there exists a $(n-1)$-dimensional convex set $\Omega'$ such that 
$$
\Omega'\supset A'\cup B',\qquad \H^{n-1}\bigl( \Omega' \setminus A'\bigr) + 
\H^{n-1}\bigl( \Omega '\setminus B' \bigr) 
\leq C\,\delta^{(\eta-\zeta)\e_{n-1}(\tau)}.
$$
Define $\zeta$ by 
\begin{equation}
\label{eq:zeta}
\zeta: =  \frac{\e_{n-1}(\tau)}{3}\eta,
\end{equation}
and set $\Omega:=\Omega'/\rho$.  Then we obtain (recall that $1/\rho\leq C$ and that $\e_{n-1}(\tau)\leq 1$)
\begin{equation}
\label{eq:induction}
\Omega\supset \A(\bar \lambda)\cup \B(\bar \lambda),\qquad \H^{n-1}\bigl( \Omega \setminus \A(\bar \lambda) \bigr) + 
\H^{n-1}\bigl( \Omega \setminus \B(\bar \lambda) \bigr) 
\leq C\,\d^{2\zeta}.
\end{equation}

\subsubsection*{Step 2: Theorem \ref{thm:n1} applies 
to most of the sets $A_y$ and $B_y$
for $y \in \A(\bar \lambda)\cap \B(\bar \lambda)$.}
Define $\C:=\A(\bar \lambda)\cap \B(\bar \lambda)\subset \S(\bar \lambda)$. By \eqref{eq:measure outside C}, \eqref{eq:induction}, \eqref{eq:normalized},
and \eqref{eq:zeta},
we have  
\begin{equation}
\label{eq:AB minus C}
\begin{split}
|A\setminus \pi^{-1}(\C)|+|B\setminus \pi^{-1}(\C)| &\leq
\bigl|A\setminus \pi^{-1}\bigl(\A(\bar \lambda)\bigr)\bigr|
+\bigl|B\setminus \pi^{-1}\bigl(\B(\bar \lambda)\bigr)\bigr|\\
&\qquad+\int_{\left(\A(\bar \lambda)\right)\setminus \left(\B(\bar\lambda)\right)}\H^1(A_y)\,dy
+\int_{\left(\B(\bar\lambda)\right)\setminus \left(\A(\bar\lambda)\right)}\H^1(B_y)\,dy\\
&\leq C\biggl(\delta^{2\zeta}+\delta^{\eta-\zeta}\biggr) + C\Bigl( \H^{n-1}\bigl( \Omega \setminus \A(\bar \lambda) \bigr) + 
\H^{n-1}\bigl( \Omega \setminus \B(\bar \lambda) \bigr) \Bigr)\\
&\leq C\Bigl(\delta^{2\zeta}+\delta^{\eta-\zeta}\Bigr)\leq C\,\delta^{2\zeta}
\end{split}
\end{equation}
Hence, by \eqref{eq:measures} and \eqref{eq:AB minus C},
\begin{equation}
\label{eq:Fubini}
\begin{split}
\int_{\C}\H^{1}\Bigl(S_y\setminus\bigl(tA_y+(1-t)B_y\bigr)\Bigr)\,dy
& = 
\int_{\C}\bigl[\H^{1}(S_y) - \H^1(tA_y+(1-t)B_y) \bigr]\,dy  \\
& \leq
\int_{\C}  \bigl[\H^{1}(S_y) - t \H^1(A_y) - (1-t)\H^1(B_y) \bigr]\,dy  \\
&= 
|S\cap \pi^{-1}(\C)|-t|A\cap \pi^{-1}(\C)| - (1-t)|B\cap \pi^{-1}(\C)|\\
&\leq |S| - t|A|-(1-t)|B|+ t|A\setminus \pi^{-1}(\C)|+(1-t)|B\setminus \pi^{-1}(\C)|\\
&\leq C\,\delta^{2\zeta}.
\end{split}
\end{equation}
Write $\C$ as $\C_1\cup \C_2$, where
$$
\C_1:=\bigl\{y \in \C: \H^1(S_y)-t\H^1(A_y)-(1-t)\H^1(B_y)\leq \delta^\zeta\bigr\},
\qquad \C_2:=\C\setminus \C_1.
$$
By Chebyshev's inequality
\begin{equation}
\label{eq:C2}
\H^{n-1}\bigl(\C_2\bigr)\leq C\,\delta^\zeta,
\end{equation}
while, recalling \eqref{eq:barlambda},
$$
\min\bigl\{\H^{1}(A_y),\H^1(B_y)\bigr\} \geq \bar \lambda >\delta^\zeta/2\qquad \forall\,y \in \C_1.
$$
Hence, by Theorem \ref{thm:n1} applied to $A_y,B_y\subset\R$ for $y \in \C_1$, we deduce that
\begin{equation}
\label{eq:freiman Ay}
\H^{1}\bigl(\co(A_y)\setminus A_y\bigr)
+\H^{1}\bigl(\co(B_y)\setminus B_y\bigr)
\leq C\,\delta^\zeta 
\end{equation}
(recall that $\co(E)$ denotes the convex hull of a set $E$).
Let $\hat \C_1\subset \C_1$ denote the set of $y \in \C_1$ such that 
\begin{equation}\label{eq:C1'}
\H^{1}\Bigl(S_y\setminus\bigl(tA_y+(1-t)B_y\bigr)\Bigr)\leq \delta^\zeta,
\end{equation}
and notice that, by \eqref{eq:Fubini} and Chebyshev's inequality,
$
\H^{n-1}(\C_1\setminus \hat \C_1)\leq C\,\delta^\zeta
$.
 Then choose a compact set ${\bar\C_1}\subset \hat \C_1$ such that $\H^{n-1}(\hat \C_1\setminus  {\bar\C_1}) \le\delta^\zeta$ to obtain
\begin{equation}
\label{eq:C11'}
\H^{n-1}(\C_1\setminus {\bar\C_1})\leq C \,\delta^\zeta.
\end{equation}
In particular, it follows from \eqref{eq:induction} that
\begin{equation}
\label{eq:C1'omega}
\H^{n-1}(\Omega\setminus {\bar\C_1})\leq C \,\delta^\zeta.
\end{equation}

\subsubsection*{Step 3: There is $\bar S \subset S$ so that $|S \setminus \bar S|$ is small and $\bar S$ is bounded.}
Define the compact sets
\begin{equation}
\label{eq:bar sets}
\bar A:= \bigcup_{y \in {\bar\C_1}} A_y,\quad
\bar B:= \bigcup_{y \in {\bar\C_1}} B_y ,\quad
\bar S:= \bigcup_{y \in {\bar\C_1}} tA_y+(1-t)B_y.
\end{equation}
Note that by \eqref{eq:normalized}, \eqref{eq:AB minus C}, \eqref{eq:C2}, 
\eqref{eq:C11'}, 
\begin{equation} 
\label{eq:barA difference}
|A\Delta \bar A| + |B\Delta \bar B| = 
|A \setminus \pi^{-1}(\bar \C_1)|
+ |B \setminus \pi^{-1}(\bar \C_1)| \le C\delta^{\zeta},
\end{equation}
therefore
\[
|\bar S| = \int_{\bar \C_1} \H^1(t A_y + (1-t)B_y) \, dy 
\ge \int_{\bar \C_1} \bigl[t \H^1( A_y) + (1-t)\H^1(B_y)\bigr] \, dy 
= t|\bar A| + (1-t)|\bar B| \ge 1 - C\d^{\zeta}.
\]
Hence, by \eqref{eq:measures}  (and $\bar S \subset S$),
\begin{equation}
\label{eq:barS difference}
|S \Delta \bar S| \le C \delta^{\zeta}
\end{equation}

Next we show that $\bar S$ is bounded. First recall that 
\begin{equation}
\label{eq:A+B}
S_y=\bigcup_{y=ty'+(1-t)y''} tA_{y'}+(1-t)B_{y''},
\end{equation}
and by \eqref{eq:C1'} we get
\begin{equation}
\label{eq:sum Ay}
\H^1\biggl(\biggl(\bigcup_{y=ty'+(1-t)y''} tA_{y'}+(1-t)B_{y''}\biggr) \setminus tA_{y}+(1-t)B_{y}\biggr) \leq  \delta^\zeta\qquad \forall\,y \in {\bar\C_1}.
\end{equation}
Recalling that $\bar \pi:\R^n \to \R$ is the orthogonal projection onto the last 
component (that is, $\bar\pi(y,s)=s$),
we define the characteristic functions
$$
\chi^A_y(s):=\left\{
\begin{array}{ll}
1 & \text{if }s \in \bar\pi(t A_y)\\
0 & \text{otherwise},
\end{array}
\right.\qquad 
\chi_y^{A,*}(s):=\left\{
\begin{array}{ll}
1 & \text{if }s \in \bar\pi(t\co(A_y))\\
0 & \text{otherwise},
\end{array}
\right.\qquad 
$$
and analogously for $B_y$ (with $1-t$ in place of $t$). Hence, by \eqref{eq:freiman Ay}
we have
the following estimate on the convolution of the functions $\chi_y$ and $\chi_y^\ast$:
\begin{equation}
\label{eq:close Linfty}
\begin{split}
\|\chi_{y'}^{A,*}\ast \chi_{y''}^{B,*} - \chi_{y'}^A\ast \chi_{y''}^B\|_{L^\infty(\R)}&\leq
\|\chi_{{y''}}^{B,*} -  \chi_{{y''}}^B\|_{L^1(\R)} + 
\|\chi_{{y'}}^{A,*} -  \chi_{{y'}}^A\|_{L^1(\R)}\\
& \leq \H^1\bigl(\co(B_{y''})\setminus B_{y''}\bigr)+\H^1\bigl(\co(A_{y'})\setminus A_{y'}\bigr)\\
&< 3\,\delta^\zeta\qquad \forall\,y',y''\in {\bar\C_1}.
\end{split}
\end{equation}
Let us denote by $[a,b]$ the interval $\bar\pi\bigl(t\co(A_{y'})+(1-t)\co(B_{y''})\bigr)$, and
notice that, since by construction
$$
\min\bigl\{t\H^1(A_y),(1-t)\H^1(B_y)\bigr\}\geq \min\{\tau,1-\tau\}\bar\lambda\geq 10\,\delta^\zeta\qquad \forall\,y \in {\bar\C_1}
$$
(see \eqref{eq:barlambda}), this interval has length greater than $20\,\delta^\zeta$. Also,
it is easy to check that
the function $\chi_{{y'}}^\ast\ast \chi_{{y''}}^\ast$ is supported on $[a,b]$, has slope equal to $1$ (resp.~$-1$)
inside $[a,a+3\,\delta^\zeta]$ (resp. $[b-3\,\delta^\zeta,b]$),
and it is greater than $3\,\delta^\zeta$ inside $[a+3\,\delta^\zeta,b-3\,\delta^\zeta]$.
Hence, since $\bar\pi\bigl(tA_{y'}+(1-t)B_{y''}\bigr)$ contains the set $\{\chi_{y'}\ast\chi_{y''}>0\}$, 
by \eqref{eq:close Linfty} we deduce that
\begin{equation}
\label{eq:A*y'}
\bar\pi\bigl(tA_{y'}+(1-t)B_{y''}\bigr) \supset [a+3\,\delta^\zeta, b-3\,\delta^\zeta],
\end{equation}
which implies in particular that
\begin{equation}
\label{eq:A+By}
\H^1\bigl(t\co(A_{y'})+(1-t)\co(B_{y''}) \bigr) \leq \H^1\bigl(tA_{y'}+(1-t)B_{y''}\bigr)+6\,\delta^\zeta\qquad \forall\,y',y''\in {\bar\C_1}.
\end{equation}
\\

We claim that if $y',y'',y=ty'+(1-t)y'' \in {\bar\C_1}$, then
\begin{equation}
\label{eq:y'y''y}
\bar\pi \bigl(t\co(A_{y'})+(1-t)\co(B_{y''})\bigr)\subset [\a_y-16\,\delta^\zeta,\b_y+16\,\delta^\zeta],
\end{equation}
where $[\a_y,\b_y]:=\bar\pi \bigl(t\co(A_{y})+(1-t)\co(B_{y})\bigr)$.

Indeed, if this was false, since 
$
\bar\pi \bigl(t\co(A_{y'})+(1-t)\co(B_{y''})\bigr)=[a,b]$ is an interval of 
length  at least $20\delta^{\zeta} \geq 16\delta^\zeta$, it follows that
$$
\H^1\bigl([a,b] \setminus  [\a_y,\b_y]\bigr)
\geq 16 \delta^\zeta.
$$
This implies that
$$
\H^1\bigl([a+3\d^\zeta, b-3\d^\zeta] \setminus 
 [\a_y,b_y]\bigr) \geq 10\d^\zeta,
$$
so, by \eqref{eq:A*y'},
$$
\H^1\bigl(\bar\pi\bigl(tA_{y'}+(1-t)B_{y''}\bigr) \setminus [\a_y,\b_y]\bigr) \geq 10 \d^\zeta.
$$
However, since $ [\a_y,\b_y]\supset \bar\pi\bigl(tA_y+(1-t)B_y\bigr)$, this contradicts \eqref{eq:sum Ay}
and proves the claim \eqref{eq:y'y''y}. \\

Now, 
if we write 
$$
\co(A_y)=\{y\}\times [a^A(y),b^A(y)],\quad \co(B_y)=\{y\}\times [a^B(y),b^B(y)],\quad \co(\bar S_y)=\{y\}\times [a^{\bar S}(y),b^{\bar S}(y)],
$$
and we denote by $c^A(y):=\frac{a^A(y)+b^A(y)}{2}$ the barycenter of $\co(A_y)$ (and analogously for $B$ and $\bar S$), 
then $b^{\bar S}=tb^A+(1-t)b^B$ and it follows from \eqref{eq:y'y''y} that
\begin{equation}
\label{eq:bS}
tb^A(y')+(1-t)b^B(y'')\leq b^{\bar S}(y)+16\,\delta^\zeta \qquad \forall\,y,y',y'' \in {\bar\C_1},\, 
y=ty'+(1-t)y''
\end{equation}
(and analogously for $a$).
Hence, from the fact that $\H^1\bigl(\co(A_y)\bigr)$ and $\H^1\bigl(\co(B_y)\bigr)$
are universally bounded (see \eqref{eq:normalized} and \eqref{eq:freiman Ay}) one easily deduces that
$$
\bigl| tc^A(y')+(1-t)c^B(y'') - c^{\bar S}(y)\bigr| \leq C \qquad \forall\,y,y',y'' \in {\bar\C_1},\, 
y=ty'+(1-t)y''.
$$
Hence, by Remark \ref{rmk:4pts} in Section \ref{sect:natural} we get that 
$$
\left|c^{A}(y_1)+c^{A}(y_2) -c^{A}(t_A'y_1+(1-t_A')y_2)-c^{A}(t_A''y_1+(1-t_A'')y_2) \right| \leq C
$$
whenever $y_1,y_2,t_A'y_1+(1-t_A')y_2,t_A''y_1+(1-t_A'')y_2 \in {\bar\C_1}$,
with $t_A':=\frac{1}{2-t}$ and $t_A'':=1-t_A'$.\\
As proved in Lemma \ref{lemma:1d}, this estimate in one dimension implies that,
along any segment $[y_1,y_2]$ on which $\bar \C_1$ has large measure
and such that $y_1,y_2,t_A'y_1+(1-t_A')y_2,t_A''y_1+(1-t_A'')y_2 \in {\bar\C_1}$, $c^{A}$ is at bounded distance from a linear function $\ell_A$.
Analogously,
$$
\left|c^{B}(y_1)+c^{B}(y_2) -c^{B}(t_B'y_1+(1-t_B')y_2)-c^{B}(t_B''y_1+(1-t_B'')y_2) \right| \leq C
$$
whenever $y_1,y_2,t_B'y_1+(1-t_B')y_2,t_B''y_1+(1-t_B'')y_2 \in {\bar\C_1}$,
now with $t_B':=\frac{1}{1+t}$ and $t_B'':=1-t'_B$ (recall again Remark \ref{rmk:4pts}),
so $c^B$ is at bounded distance from a linear function $\ell_B$
along any segment $[y_1,y_2]$ on which $\bar \C_1$ has large measure
and such that $y_1,y_2,t_B'y_1+(1-t_B')y_2,t_B''y_1+(1-t_B'')y_2 \in {\bar\C_1}$.
Hence, along any segment 
$[y_1,y_2]$ on which $\bar \C_1$ has large measure
and such that $y_1,y_2,t_A'y_1+(1-t_A')y_2,t_A''y_1+(1-t_A'')y_2,t_B'y_1+(1-t_B')y_2,t_B''y_1+(1-t_B'')y_2 \in {\bar\C_1}$,
$c^{\bar S}$ is at bounded distance from the linear function $\ell(y):=t\ell_A(y)+(1-t)\ell_B(y)$.\\

We now use this information to deduce that, up to an affine transformation of the form
\begin{equation}
\label{eq:affine}
\R^{n-1} \times \R \ni (y,s) \mapsto (Ty,t-Ly)+(y_0,t_0)
\end{equation}
with $T:\R^{n-1}\to \R^{n-1}$,
 $\det(T)=1$, and $(y_0,t_0) \in \R^n$, the set $\bar S$ is universally bounded, say $\bar S\subset B_R$ for some universal constant $R$.
 
 Indeed, first of all, 
 since ${\bar\C_1}$
is almost of full measure inside the convex set $\Omega$ (see \eqref{eq:induction}, \eqref{eq:C2},
and \eqref{eq:C11'}),
 by a simple Fubini argument (see the analogous argument in \cite[Proof of Theorem 1.2, Step 3-b]{fjA}
 for more details)
 we can choose $n$ ``good'' points
 $y_1,\ldots,y_n \in \bar \C_1$ such that:
 \begin{enumerate}
\item[(a)] All points
$$
y_1,\ldots,y_n\quad \text{and} \quad t_A'y_1+(1-t_A')y_2,t_A''y_1+(1-t_A'')y_2,t_B'y_1+(1-t_B')y_2,t_B''y_1+(1-t_B'')y_2
$$
belong to $\bar \C_1$.
\item[(b)] Let $\Sigma_i$, $i=1,\, \dots, n$, denote the $(i-1)$-dimensional 
simplex generated by $y_1,\ldots, y_i$, 
and define
$$
\Sigma_i':=\bigl[t_A'\Sigma_i + (1-t_A')y_{i+1}\bigr]\cup 
\bigl[t_A''\Sigma_i + (1-t_A'')y_{i+1}\bigr]\cup \bigl[t_B'\Sigma_i + (1-t_B')y_{i+1}\bigr]\cup 
\bigl[t_B''\Sigma_i + (1-t_B'')y_{i+1}\bigr] ,
$$
$i=1,\, \dots,\,  n-1$. 
Then
\begin{enumerate}
\item[(i)]
$$
\H^{i-1}(\Sigma_i) \geq c_n,\quad \frac{\H^{i-1}(\Sigma_i\cap \bar \C_1)}{\H^{i-1}(\Sigma_i)} \geq 1-\delta^{\zeta/2}, 
\qquad \forall\,i=2,\ldots,n;
$$
\item[(ii)]
$$
\frac{\H^{i-1}\left(\Sigma_i'\cap \bar \C_1 \right)}{\H^{i-1}\left(\Sigma_i'\right)} \geq 1-\delta^{\zeta/2}\qquad \forall\,i=2,\ldots,n-1.
$$
\end{enumerate}
\end{enumerate}
Then, thanks to John's Lemma \cite{john}, up to an affine transformation of the form \eqref{eq:affine} we can assume that
\begin{equation}
\label{eq:horiz}
B_{r} \subset \Omega\subset B_{(n-1)r},\qquad 1/C_n < r < C_n \text{ with $C_n$
dimensional},
\end{equation}
and 
\begin{equation}
\label{eq:vert}
(y_k,0) \in \bar S, \qquad \forall\, k= 1, \ldots, n \, .
\end{equation}
We then prove that $\bar S$ is universally bounded as follows:
first of all, thanks to \eqref{eq:horiz} we only need to show that $\bar S$
is bounded in the last variable.
Then, by \eqref{eq:vert} and (a) and (b)-(i) above, thanks to Lemma \ref{lemma:1d}
we deduce that $c^{\bar S}$ is universally bounded on $\Sigma_2\cap \bar \C_1$.

One then iterates this construction:
since $c^{\bar S}$
is universally bounded on $\Sigma_2\cap \bar\C_1'$ and at $y_3$, for any point $z \in \Sigma_2\cap \bar\C_1$ such that
$t_A'z+(1-t_A')y_3,t_A''z+(1-t_A'')y_3,t_B'z+(1-t_B')y_3,t_B''z+(1-t_B'')y_2 \in \Sigma_2'\cap \bar\C_1$ (these are  most of the points)
we can apply again Lemma \ref{lemma:1d} 
to deduce that $c^{\bar S}$ is universally bounded on the set $[z,y_3]\cap \bar \C_1$.
An iteration of this argument as in \cite[Proof of Theorem 1.2, Step 3-d]{fjA}
shows that $c^{\bar S}$ is universal bounded on a set $\Sigma_n''$
such that $\H^{n-1}(\Sigma_n\setminus \Sigma_{n}'')\leq C\,\delta^{\zeta/2}$.

We now conclude as in \cite[Proof of Theorem 1.2, Step 3-e]{fjA}:
fix a point $\bar y_1 \in \bar \C_1$. Then we can find another point $\bar y_2
\in \bar \C_1$ such that 
$t_A'\bar y_1+(1-t_A')\bar y_2,t_A''\bar y_1+(1-t_A'')\bar y_2,
t_B'\bar y_1+(1-t_B')\bar y_2,t_B''\bar y_1+(1-t_B'')\bar y_2 \in \bar \C_1$,
most of the points on the segment $[\bar y_1,\bar y_2]$ belong to $\bar \C_1$,
and $\H^1\bigl([\bar y_1,\bar y_2]\cap \Sigma_{n}'' \bigr) \geq c_n'$
for some dimensional constant $c_n'>0$.
Hence, on this segment $c^{\bar S}$ must be at some bounded distance from a linear function
$\ell$, but at the same time we know that $c^{\bar S}$ is universally bounded on 
$[\bar y_1,\bar y_2]\cap \Sigma_{n}''$, so $\ell$ is universally bounded there. Since
this set has non-trivial measure, this implies that $\ell$ has to be universally bounded
on the whole segment $[\bar y_1,\bar y_2]$ (since $\ell$ is a linear function).
Thus $c^{\bar S}$ is universally bounded on $[\bar y_1,\bar y_2]\cap  \bar \C_1$ as well,
and this provides a universal bound for $c^{\bar S}(\bar y_1)$, concluding the proof.

\subsubsection*{Step 4: There are uniformly bounded 
vertically convex sets $A^\sim$ and $B^\sim$ near
$A$ and $B$.}
Let $\bar A$, $\bar B$, and $\bar S$ be as in \eqref{eq:bar sets}, and recall that 
by the previous step there exists a constant $R$ such that
$\bar S\subset \{|x_n|\leq R\}$. Let us apply opposite translations 
along the $e_n$-axis to $t\bar A$ and $(1-t)\bar B$ (see \eqref{eq:bar sets}), i.e.,
$$
t\bar A \mapsto t\bar A+\mu e_n,\qquad (1-t) \bar B \mapsto (1-t)\bar B-\mu e_n,
$$
for some $\mu \in \R$, so that
$\bar A\subset \{x_n \geq -R\}$ and $\bar A\cap \{x_n = -R\}\neq \emptyset$ (recall that $\bar A$ is compact).
This means that
$$
\min_{y \in {\bar\C_1}}a^A(y)=-R.
$$
Notice that, thanks to \eqref{eq:y'y''y},
\begin{equation}
\label{eq:aS}
ta^A(y')+(1-t)a^B(y'')\geq a^{\bar S}(y)-16\,\delta^\zeta \qquad \forall\,y,y',y'' \in {\bar\C_1},\, 
y=ty'+(1-t)y''.
\end{equation}
Let $\bar y \in {\bar\C_1}$ be such that $a^A(\bar y)=-R$, and set
$$
\C_1^-:= {\bar\C_1}\cap \frac{{\bar\C_1}-t\bar y}{1-t}.
$$
Then, since 
$a^A(\bar y)=-R$ and $a^{\bar S}\geq -R$, it follows from \eqref{eq:aS} that
$$
a^B(y'') \geq -R-C\,\delta^\zeta \geq -R-1 \qquad \forall\, y'' \in  \C_1^-.
$$
Define
\[
A^\sim := \bigcup_{y \in \C_1^-} \{y\} \times [a^A(y),b^A(y)],
\qquad 
B^\sim := \bigcup_{y \in \C_1^-} \{y\} \times [a^B(y),b^B(y)].
\]
We have shown that 
\[
A^\sim \cup B^\sim \subset \{x_n \geq -R-1\}.
\]
\\

It remains to prove the upper bounds. 
Note that because $\bar y\in \Omega$ and $\Omega$ is convex, 
it follows from \eqref{eq:C1'omega} that
\[
\H^{n-1}\left(\Omega \setminus \frac{\bar \C_1 - t\bar y}{1-t}\right)
= (1-t)^{1-n} \H^{n-1}\bigl(((1-t)\Omega + t\bar y)\setminus \bar \C_1\bigr)
\le
(1-t)^{1-n} \H^{n-1}(\Omega \setminus \bar \C_1) \le C\,\delta^\zeta.
\]
Therefore, using \eqref{eq:C1'omega} again, we have 
\begin{equation}
\label{eq:meas C1-}
\H^{n-1} (\Omega \setminus \C_1^-) \leq C\,\delta^\zeta.
\end{equation}

We now claim that $A^\sim\cup B^\sim \subset \{x_n \leq \tilde CR\}$ for 
some universal constant $\tilde C$.  
Indeed, if for instance $b^A(\tilde y)\geq \tilde CR$ for some 
$\tilde y \in {\C_1^-}$, then we could use \eqref{eq:bS}
and the fact that $b^B\geq a^B\geq -R-1$ on  $\C_1^-$ to get
$$
b^{\bar S}(y)\geq t\tilde CR-(1-t)(R+1)- 16\,\delta^\zeta \geq \tau \tilde CR-R-2 \qquad \forall\,y \in {\bar\C_1}\cap (t \tilde y + (1-t)\C_1^-),
$$
and since the latter set is nonempty (because of \eqref{eq:C1'omega}, \eqref{eq:meas C1-} and the convexity of $\Omega$) this contradicts the fact that 
$b^{\bar S} \leq R$ provided $\tilde C$ is large enough (the case
$b^B(\tilde y)\geq \tilde CR$ for some $\tilde y \in {\bar\C_1}$ is completely analogous).  Thus, $A^\sim$ and $B^\sim$ are universally bounded. \\

Finally, note that \eqref{eq:normalized}, \eqref{eq:barA difference}, 
\eqref{eq:meas C1-}, and \eqref{eq:freiman Ay} imply
\begin{equation} 
\label{eq:Asim difference}
|A\Delta A^\sim| + |B\Delta B^\sim| \le C\delta^{\zeta}.
\end{equation}

\subsubsection*{Step 5: The inductive hypothesis applies to 
horizontal sections and hence there are convex sets 
close to $A^\sim$ and $B^\sim$.}

The main goal of this section is to show that the hypotheses of Lemma 
\ref{lemma:concavity}
apply to the function $b^A$ (and similarly to $b^B$, $-a^A$, and $-a^B$).  
The fact that $A^\sim$ and $B^\sim$ are close to convex sets will then follow
easily. 

As explained in the outline of the proof in Section \ref{sect:outline},
to be able to apply Lemma \ref{lemma:concavity} we will construct
auxiliary sets $A^-$ and $B^-$ which consist of
the top profile of $A^\sim$ and $B^\sim$ with a flat bottom,
for which the slices coincide with the superlevel sets of $b^A$ and $b^B$,
and we will apply Lemma \ref{lemma:BM n n-1} to such sets.
However, to be able to do this, we must show that $A^-$ and $B^-$
are almost optimal in the Brunn-Minkowski inequality.
\\

As we showed in Step 4, $A^\sim$ and $B^\sim$ are universally bounded, 
so we may choose universal constants $M_A\ge0$ and $M_B\ge0$ such that 
$$
-M_A \leq a^A(y),\quad -M_B \leq a^B(y)\qquad \forall\, y \in \mathcal \C_1^-,
$$
and such that the sets 
\[
A^-:=\bigcup_{y \in \mathcal \C_1^-}\{y\}\times [-M_A, b^A(y)],
\qquad B^-:=\bigcup_{y \in \mathcal \C_1^-} \{y\} \times [-M_B,b^B(y)],
\]
are universally bounded.  We may also adjust the constants $M_A$ and $M_B$ so
that $|A^-| = |B^-|$. 

Define 
\[
S^-:=t A^-+(1-t)B^-; \quad 
\C^-(y) := \{(y',y'')\in \C^- \times \C^-: ty' + (1-t)y''= y\}.
\]
We estimate the measure of $S^-$ using \eqref{eq:bS} as follows:
\begin{align*}
|S^-|
&=\int_{t\C_1^-+(1-t)\C_1^-}
\H^1\biggl(\bigcup_{(y',y'')\in \C^-(y)}  
t[-M_A, b^A(y')]+(1-t)[-M_B, b^B(y'')]\biggr)\,dy\\
&\leq 
\int_{t\C_1^-+(1-t)\C_1^-}
\Bigl(b^{\bar S}(y) + 16 \,\delta^\zeta + t M_A + (1-t)M_B\Bigr) \, dy \\
& 
\leq
\int_{\C_1^-}
\Bigl(b^{\bar S}(y) + t M_A + (1-t)M_B\Bigr) \, dy  + C\,\delta^\zeta,
\end{align*}
where, in the final inequality, we used that
$\H^{n-1}\bigl((t\C_1^-+(1-t)\C_1^-)\setminus \C_1^- \bigr)\leq C\,\delta^\zeta$
(recall that $\C_1^-\subset \Omega$ and $\Omega$ is convex,
thus $t\C_1^-+(1-t)\C_1^-\subset \Omega$ and the bound follows from \eqref{eq:meas C1-})
and that 
$b^{\bar S}$ is universally bounded on $\C_1^-$.  Next, since
$b^{\bar S}=tb^{A}+(1-t)b^B$ and $|A^-| = |B^-|$, it follows that 
\begin{align*}
|S^-| & \leq 
t \int_{\C_1^-}(b^A(y)- M_A) \, dy +
(1-t) \int_{\C_1^-}(b^B(y)- M_B) \, dy + C\,\delta^\zeta \\
&= t |A^-|+(1-t)|B^-| +C\,\delta^\zeta = |A^-|+C\,\delta^\zeta,
\end{align*}
On the other hand \eqref{eq:BM} implies
$$
|S^-|\geq \bigl(t|A^-|^{1/n}+(1-t)|B^-|^{1/n}\bigr)^n=|A^-|.
$$
Hence, in all, we find that 
\begin{equation}
\label{eq:new measure S-}
0 \le  |S^-| - |A^-|  \le C\,\delta^\zeta \quad \mbox{and} \quad |A^-| = |B^-|
\end{equation}
We are now in a position to apply Lemma \ref{lemma:BM n n-1} to $A^-$ and $B^-$
to confirm that hypothesis \eqref{eq:level sets co} of Lemma \ref{lemma:concavity} 
is valid for $b^A$ and $b^B$.\\

Let us recall the notation 
$E(s)\subset \R^{n-1}\times\{s\}$ in \eqref{eq:Ey}. Since $|cA^-|= |cB^-| = 1$
for some universal constant $c>0$, by applying 
\eqref{eq:new measure S-} and Lemma \ref{lemma:BM n n-1}
 to the sets $cA^-$, $cB^-$, and $cS^-$, we find a monotone map 
$T:\R\to \R$ such that
$$
T_\sharp\rho_{A^-}=\rho_{B^-},\qquad \rho_{A^-}(s):=\frac{\H^{n-1}\bigl(A^-(s)\bigr)}{|A^-|},\quad \rho_{B^-}(s):=\frac{\H^{n-1}\bigl(B^-(s)\bigr)}{|B^-|},
$$
\begin{equation}
\label{eq:T'} 
 T'(s)=\frac{\H^{n-1}\bigl(A^-(s)\bigr)|B^-|}{\H^{n-1}\bigl(B^-(T(s))\bigr)|A^-|}\qquad \rho_A\text{-a.e.},
\end{equation}
\begin{equation}
\label{eq:BM A-}
\int_\R e_{n-1}(s)\,[t+(1-t)T'(s)]\,ds
\leq C\,\delta^\zeta,
\end{equation}
and
\begin{equation}
\label{eq:meas slices}
\int_\R \biggl| \frac{\rho_{A^-}(s)}{\rho_{B^-}(T(s))} - 1\biggr|\,\rho_{A^-}(s)\,ds \leq C\,\delta^{\zeta/2},
\end{equation}
where $T_t(s)=ts+(1-t)T(s)$ and
$$
e_{n-1}(s) : = \H^{n-1}\bigl(S^-(T_t(s))\bigr)-
\left[t\H^{n-1}\bigl(A^-(s)\bigr)^{1/(n-1)}+(1-t)\H^{n-1}\bigl(B^-(T(s))\bigr)^{1/(n-1)} \right]^{n-1}.
$$
Let us define the set
\begin{equation}
\label{eq:G}
G:=\biggl\{s \in \R\,:\,e_{n-1}(s)\leq \delta^{\zeta/2}\biggr\},
\end{equation}
and observe that, thanks to \eqref{eq:BM A-}, \eqref{eq:G}, and $T' \ge 0$,
\begin{equation}
\label{eq:meas G}
\H^1(\R\setminus G) \leq \frac{1}{\tau} \int_{\R\setminus G} [t+(1-t)T'(s)]\,ds \leq C\,\delta^{\zeta/2}.
\end{equation}
Next, note that the formula for $T$ (with $A$ and $B$ replaced by $A^-$ and $B^-$)
given in the footnote in the statement of Lemma \ref{lemma:BM n n-1} implies that the distributional 
derivative of $T$ has no singular part on $T^{-1}(\{\rho_{B^-}>0\})$.  Hence,
the area formula gives
$$
\H^1\bigl(\bigl(\R\setminus T(G)\bigr) \cap \{\rho_B>0\}\bigr) =  \int_{\left(\R\setminus G\right) \cap T^{-1}(\{\rho_B>0\})} T'(s)\,ds,
$$
and it follows that 
\begin{equation}
\label{eq:meas T G}
\H^1\bigl(\bigl(\R\setminus T(G)\bigr) \cap \{\rho_{B^-}>0\}\bigr) 
\leq  \int_{\R\setminus G} T'(s)\,ds \leq \frac{1}{\tau} \int_{\R\setminus G} [t+(1-t)T'(s)]\,ds \leq C\,\delta^{\zeta/2}.
\end{equation}
Also, we define
$$
I_{A^-}:=\Bigl\{s \in \R\,:\, \H^{n-1}\bigl(A^-(s)\bigr)>\delta^{\zeta/4}\Bigr\},\qquad
I_{B^-}:=\Bigl\{s \in \R\,:\, ,\,\H^{n-1}\bigl(B^-(s)\bigr)>\delta^{\zeta/4} \Bigr\},
$$
and
$$
I_T:=\biggl\{s \in \R\,:\, \ \frac{2}3 \leq \frac{\rho_{A^-}(s)}{\rho_{B^-}(T(s))} \leq \frac{3}2\biggr\}.
$$
Notice that, thanks to \eqref{eq:meas slices},
$$
\int_{I_T}\rho_{A^-}(s)\,ds \geq 1-C\,\delta^{\zeta/2}.
$$
Also, since $\rho_{A^-}$ and  $\rho_{B^-}$ are probability densities supported inside some bounded interval (being $A^-$ and $B^-$ universally bounded), we have
$$
\int_{I_{A^-}}\rho_{A^-}(s)\,ds= \int_{\{\rho_{A^-}>\delta^{\zeta/4}|A^-|\}}\rho_{A^-}(s)\,ds \geq 1-C\,\delta^{\zeta/4}
$$
and (using the condition $T_\sharp\rho_{A^-}=\rho_{B^-}$)
$$
\int_{T^{-1}(I_{B^-})}\rho_{A^-}(s)\,ds = \int_{I_{B^-}}\rho_{B^-}(s)\,ds =\int_{\{\rho_{B^-}>\delta^{\zeta/4}|B^-|\}}\rho_{B^-}(s)\,ds \geq 1-C\,\delta^{\zeta/4}.
$$
Therefore
\begin{equation}
\label{eq:level sets A-}
\int_I\rho_{A^-}(s)\,ds=\int_{T(I)}\rho_{B^-}(s)\,ds \geq 1-C\,\delta^{\zeta/4},\qquad I:=I_{A^-}\cap T^{-1}(I_{B^-})\cap I_T.
\end{equation}
We now apply the inductive hypothesis to $A^-(s)$, $B^-(T(s))$, $S^-(T_t(s))$:
define
$$
A^s:=\frac{A^-(s)}{H^{n-1}\bigl(A^-(s)\bigr)^{1/(n-1)}},\quad B^s:=\frac{B^-(T(s))}{H^{n-1}\bigl(B^-(T(s))\bigr)^{1/(n-1)}},
$$
$$
S^s:=\frac{S^-(T_t(s))}{tH^{n-1}\bigl(A^-(s)\bigr)^{1/(n-1)}+(1-t)H^{n-1}\bigl(B^-(T(s))\bigr)^{1/(n-1)}},
$$
$$
t^s:=\frac{tH^{n-1}\bigl(A^-(s)\bigr)^{1/(n-1)}}{tH^{n-1}\bigl(A^-(s)\bigr)^{1/(n-1)}+(1-t)H^{n-1}\bigl(B^-(T(s))\bigr)^{1/(n-1)}}.
$$
Then, since 
$\frac{2}3 \leq \frac{\rho_{A^-}(s)}{\rho_{B^-}(T(s))} \leq \frac{3}2$ for $s\in I$, 
and $|A^-| = |B^-|$, it follows that 
$$
t^s \in \biggl[\frac{\tau}{2}, 1- \frac{\tau}{2}\biggr]\qquad \forall\,s \in I.
$$
In addition, recalling the definition of $G$, for any $s \in I\cap G$ we also have
$$
S^s=t^sA^s+(1-t^s)B^s,\qquad \H^{n-1}(A^s)=\H^{n-1}(B^s)=1,\qquad 
\H^{n-1}(S^s) 
\leq 1 + \delta^{\zeta/4}.
$$
Hence, if 
\begin{equation}
\label{eq:cond omega'}
\d^{\zeta/4} \leq e^{-M_{n-1}(\tau/2)},
\end{equation}
then by the inductive hypothesis we deduce the existence of a convex set 
$K^s$ such that, up to a translation (which may depend on $s$) 
$$
K^s\supset A^s\cup B^s,\qquad \H^{n-1}\bigl(K^s\setminus A^s\bigr)+\H^{n-1}\bigl(K^s\setminus B^s\bigr)\leq 
C\,\delta^{\frac\zeta4 \e_{n-1}(\tau/2)}.
$$
Thus, in particular,
$$
\H^{n-1}\bigl(\co(A^s)\setminus A^s\bigr)+ \H^{n-1}\bigl(\co(B^s)\setminus B^s\bigr) \leq \d^{\frac\zeta4 \e_{n-1}(\tau/2)},
$$
which implies that
$$
\H^{n-1}\bigl(\co\bigl(A^-(s)\bigr)\setminus A^-(s)\bigr)+
\H^{n-1}\bigl(\co\bigl(B^-(T(s))\bigr)\setminus B^-(T(s))\bigr)\leq  
\d^{\frac\zeta4 \e_{n-1}(\tau/2)}
\qquad \forall\,s \in I\cap G.
$$
Hence, integrating with respect to $s \in I\cap G$ and using that 
$T'\leq C$ on $I\cap G$ (as a consequence of \eqref{eq:T'} 
and the fact that $I\subset I_T$) we obtain
\begin{equation}
\label{eq:level sets A-2}
\int_{I\cap G} \H^{n-1}\bigl(\co\bigl(A^-(s)\bigr)\setminus A^-(s)\bigr)\,ds\leq  C\,\d^{\frac\zeta4\e_{n-1}(\tau/2)},
\end{equation}
\begin{equation}
\label{eq:level sets B-2}
\begin{split}
\int_{T(I\cap G)} \H^{n-1}\bigl(\co\bigl(B^-(s)\bigr)\setminus B^-(s)\bigr)\,ds&=
\int_{I\cap G} \H^{n-1}\bigl(\co\bigl(B^-(T(s))\bigr)\setminus B^-(T(s))\bigr)\,T'(s)\,ds\\
&\leq  C\,\d^{\frac\zeta4\e_{n-1}(\tau/2)}.
\end{split}
\end{equation}
Also, recalling the definition of $\rho_{A^-}$ and $\rho_{B^-}$, it follows from \eqref{eq:level sets A-}, \eqref{eq:meas G}, \eqref{eq:meas T G}, that
\[
\int_{\R\setminus (I\cap G)} \H^{n-1}(A^-(s))\,ds+
\int_{\R\setminus \left(T(I\cap G)\right)} 
\H^{n-1}(B^-(s)) \,ds\leq  C\,\delta^{\zeta/4}
\]
(notice that $B^-(s)=\emptyset$ on $\{\rho_B=0\}$). \\

By the bound above, \eqref{eq:bS}, \eqref{eq:aS}, \eqref{eq:level sets A-},  \eqref{eq:level sets A-2}, \eqref{eq:level sets B-2},
and Remark \ref{rmk:4pts} (see Section \ref{sect:natural}),
we can apply Lemma \ref{lemma:concavity} to $b^A$ 
find a concave function $\Psi^+(y)$ defined on $\Omega$ such that
\[
\int_{\C_1^-} |b^A(y) - \Psi^+(y)| \, dy \le 
C\,\d^{\frac{\zeta\,\beta_{n,\tau}}{4} \e_{n-1}(\tau/2)}.
\]
Similarly, there is a convex function $\Psi^-$ on $\Omega$ such that
\[
\int_{\C_1^-} |a^A(y) - \Psi^-(y)| \, dy \le 
C\,\d^{\frac{\zeta\,\beta_{n,\tau}}{4} \e_{n-1}(\tau/2)},
\]
so the convex set 
\[
K_A := \bigl\{(y,s)\,:\,  y\in \Omega, \  \Psi^-(y) \le s \le \Psi^+(y)\bigr\}
\]
satisfies  $|A^\sim \Delta K_A| \le C\,\d^{\frac{\zeta\,\beta_{n,\tau}}{4} \e_{n-1}(\tau/2)}$.
The same argument also applies to $B^-$ so that, in all, we have 
\begin{equation}
\label{eq:new KA KB}
|A^\sim\Delta K_A|+ |B^\sim\Delta K_B| \leq 
C\,\d^{\frac{\zeta\,\beta_{n,\tau}}{4} \e_{n-1}(\tau/2)}.
\end{equation}

\subsubsection*{Step 6: Conclusion.}  

By \eqref{eq:new KA KB} and \eqref{eq:Asim difference},
we can apply Proposition \ref{prop:coS} to deduce that,
up to a translation, there exists a convex set $\K$ convex such that $A\cup B\subset \K$ and 
\begin{equation}
\label{eq:final bound}
|\K\setminus A|+|\K\setminus B|\leq C\,\d^{\frac{\zeta\,\beta_{n,\tau} }{8\,n^3}\e_{n-1}(\tau/2)},
\end{equation}
concluding the proof.

\subsubsection*{Step 7: An explicit bound for $\e_n(\tau)$ and $M_n(\tau)$.}
By \eqref{eq:final bound} and \eqref{eq:cond omega'} it follows that the 
recurrence for $\e_n(\tau)$ and  $M_n(\tau)$  is given, respectively, by
\[
\e_n(\tau) = \frac{\zeta\, \beta_{n,\tau}}{8\,n^2}\,\e_{n-1}(\tau/2),\qquad 
M_n(\tau) = \frac{4 }{\zeta}\,M_{n-1}(\tau/2).
\]
Recall that (see \eqref{eq:zeta} and Lemma \ref{lemma:concavity})
\[
\zeta=  \frac{\e_{n-1}(\tau)}{3}\eta,\qquad 
\beta_{n,\tau}= \frac{\tau}{16(n-1)|\log(\tau)|}.
\]

For $n=1$, Theorem \ref{thm:n1} implies that if $\delta < \tau/2$, then
\[
|\K\setminus A| + |\K \setminus B| \le 8\delta/\tau.
\]
In other  words, $\e_1(\tau) = 1$ and $M_1(\tau) = |\log (\tau/3)|$ are admissible choices.  

For $n \geq 2$ we have (see \eqref{eq:beta1} and \eqref{eq:eta})
\[
\eta=\frac{\bar\alpha}{n^2} = \frac{\beta_{n,\tau}}{2^4n^2} = \frac{\tau}{2^{7}(n-1)n^2 |\log \tau|},
\]
thus
\[
\zeta= 
\frac{\tau}{2^{7}\cdot 3(n-1)n^2\,|\log\tau|}\,\e_{n-1}(\tau),
\]
which gives
\[
\e_n(\tau) = \frac{\tau^2}{2^{14}\cdot 3(n-1)^2n^4 |\log \tau|^2}\,\e_{n-1}(\tau) \e_{n-1}(\tau/2) .
\]
In particular, we obtain and explicit lower bound for all $n\ge 2$ (which can be easily checked
to hold by induction):
\[
\e_{n}(\tau) \ge \frac{\tau^{3^n}}{2^{3^{n+1}} n^{3^n} |\log \tau|^{3^n}}.
\]
Concerning $M_n(\tau)$ we have
\[
M_n(\tau) = \frac{4 }{\zeta}\,M_{n-1}(\tau/2)=\frac{2^{9}\cdot 3(n-1)n^2\,|\log\tau|}{\tau\,\e_{n-1}(\tau)}\, M_{n-1}(\tau/2)
\]
from which we get
\[
M_{n}(\tau) \le \frac{2^{3^{n+2}} n^{3^n} |\log \tau|^{3^n}}{\tau^{3^n}}.
\]

\section{Proof of the technical results}
\label{sect:proofs}

As in the previous section, we use $C$ to denote a generic constant,
which may change from line to line, and that is bounded from above by $\tau^{-N_n}$
for some dimensional constant $N_n>1$.
Again, we will say that such a constant is universal.

\subsection{
Proof of Lemma \ref{lemma:BM n n-1}: an inductive proof of the Brunn-Minkowski inequality}
\label{sect:BM induction}

In this section we show how to prove the Brunn-Minkowski inequality \eqref{eq:BM} by induction on dimension.\footnote{The one-dimensional case is elementary,
and can be proved for instance as follows: given $A,B\subset \R$ compact, 
after translation we can assume that
$$
A\subset (-\infty,0],\qquad B \subset [0,+\infty),\qquad A\cap B=\{0\}.
$$
Then $A+B\supset A\cup B$, hence $|A+B| \geq |A\cup B|= |A|+|B|$, as desired.}
As a byproduct of our proof we obtain the 
bounds \eqref{eq:BM n n-1} and \eqref{induction:key}.

Given compact sets $A,B\subset\R^n$ and $S:=tA+(1-t)B$, we define the 
probability densities on the real line $\rho_A$, $\rho_B$, and $\rho_S$
as in \eqref{eq:rhoABS}, 
and we let $T:\R\to \R$ be the monotone rearrangement sending $\rho_A$ onto $\rho_B$.
Since monotone functions are differentiable almost everywhere, as a consequence of the Area Formula one has (see for instance \cite[Lemma 5.5.3]{AGS})
\begin{equation}
\label{eq:jacobian}
T'(s)=\frac{\rho_A(s)}{\rho_B(T(s))}\qquad \rho_A\text{-a.e.}
\end{equation}
Set $T_t(s):=ts+(1-t)T(s)$ 
and observe that $S(T_t(s))\supset tA(s)+(1-t)B(T(s))$, so by the Brunn-Minkowski inequality in $\R^{n-1}$ we get
\begin{equation}
\label{eq:BM Tt} 
\H^{n-1}\bigl(S(T_t(s))\bigr)^{1/(n-1)}\geq  t\H^{n-1}\bigl(A(s)\bigr)^{1/(n-1)}+(1-t)\H^{n-1}\bigl(B(T(s))\bigr)^{1/(n-1)}.
\end{equation}
We now write
\begin{align*}
|S|&=|S|\int_\R \rho_S(s)\,ds \geq |S|\int_\R\rho_S(T_t(s))\,T_t'(s)\,ds\\
&=\int_\R\H^{n-1}\bigl(S(T_t(s))\bigr)\,[t+(1-t)T'(s)]\,ds,
\end{align*}
(Here, when we applied the change of variable $s \mapsto T_t(s)$, 
we used the fact that since $T_t$ is increasing, its pointwise derivative
is bounded from above by its distributional derivative.)

Define
\begin{equation}
\label{eq:mu i}
\mu_1(s)=\ldots=\mu_{n-1}(s):=\frac{1-t}{t}\frac{|B|^{1/(n-1)}\rho_B(T(s))^{1/(n-1)}}{|A|^{1/(n-1)}\rho_A(s)^{1/(n-1)}},\qquad \mu_n(s):=\frac{1-t}{t}\frac{\rho_A(s)}{\rho_B(T(s))}
\end{equation}
Using \eqref{eq:BM Tt} and \eqref{eq:jacobian}, we obtain 
\begin{align*}
|S|&\geq \int_\R \Bigl(t\H^{n-1}\bigl(A(s)\bigr)^{1/(n-1)}+(1-t)\H^{n-1}\bigl(B(T(s))\bigr)^{1/(n-1)}\Bigr)^{n-1}\,[t+(1-t)T'(s)]\,ds\\
&=\int_\R \Bigl(t|A|^{1/(n-1)}\rho_A(s)^{1/(n-1)}+(1-t)|B|^{1/(n-1)}\rho_B(T(s))^{1/(n-1)}\Bigr)^{n-1}\,\biggl(t+(1-t)\frac{\rho_A(s)}{\rho_B(T(s))}\biggr)\,ds\\
&=|A|\int_\R t^n \prod_{i=1}^n (1+ \mu_i(s))\, 
\rho_A(s)\,ds.
\end{align*}
We now use the following inequality, see \cite[Equation (22)]{fmpBM} and \cite[Lemma 2.5]{fmpK}: there exists a dimensional constant $c(n)>0$ such that, for any choice of nonnegative numbers $\{\mu_i\}_{i=1,\ldots,n}$,
$$
\prod_{i=1}^n (1+\mu_i) 
\geq \biggl(1+\biggl(\prod_{i=1}^n \mu_i \biggr)^{1/n} \biggr)^n
+c(n) \frac{1}{\max_{i}\mu_i}\sum_{j=1}^n\biggl(\mu_j-\Bigl(\prod_{i=1}^n \mu_i \Bigr)^{1/n}\biggr)^2.
$$
Hence, we get
\begin{equation*}
\prod_{i=1}^n (1+\mu_i(s)) 
\geq \biggl(1+\frac{1-t}{t}\frac{|B|^{1/n}}{|A|^{1/n}} \biggr)^n
+c(n) \frac{1}{\max_{i}\mu_i(s)}\sum_{j=1}^n\biggl(\mu_j(s)-\frac{1-t}{t}\frac{|B|^{1/n}}{|A|^{1/n}} \biggr)^2,
\end{equation*}
which gives (recall that $\int \rho_A =1$)
\begin{align*}
|S|&\geq |A|\int_\R t^n\biggl(1+\frac{1-t}{t}\frac{|B|^{1/n}}{|A|^{1/n}} \biggr)^n\,\rho_A(s)\,ds\\
&\qquad +c(n)|A|t^n\int_\R \frac{1}{\max_{i}\mu_i(s)}\sum_{j=1}^n\biggl(\mu_j(s)-\frac{1-t}{t}\frac{|B|^{1/n}}{|A|^{1/n}} \biggr)^2\,\rho_A(s)ds\\
&\geq \Bigl(t|A|^{1/n}+(1-t)|B|^{1/n}\Bigr)^n,
\end{align*}
which proves the validity of Brunn-Minkowski in dimension $n$.
As a byproduct of this proof we will deduce \eqref{eq:BM n n-1} and \eqref{induction:key}.

Indeed \eqref{eq:BM n n-1} is immediate from
our proof.  Moreover, 
we have
\begin{equation}
\label{eq:stability T}
|S|-\Bigl(t|A|^{1/n}+(1-t)|B|^{1/n}\Bigr)^n\geq
c(n)|A|t^n\int_\R \frac{1}{\max_{i}\mu_i(s)}\sum_{j=1}^n\biggl(\mu_j(s)-\frac{1-t}{t}\frac{|B|^{1/n}}{|A|^{1/n}} \biggr)^2\rho_A(s)\,ds.
\end{equation}
With the further assumption \eqref{eq:measures}, \eqref{eq:stability T} gives
\begin{align*}
&\int_\R \frac{1}{\max_{i}\mu_i(s)}\biggl(\sum_{j=1}^n\Bigl|\frac{t}{1-t}\mu_j(s)-\frac{|B|^{1/n}}{|A|^{1/n}} \Bigr|\biggr)^2\rho_A(s)\,ds
\\
&\leq n\int_\R \frac{1}{\max_{i}\mu_i(s)}
\sum_{j=1}^n\biggl(\frac{t}{1-t}\mu_j(s)-\frac{|B|^{1/n}}{|A|^{1/n}} \biggr)^2\rho_A(s)\,ds\\
& \leq \frac{C(n)}{t^{n-2}(1-t)^{2}} \delta \leq  \frac{C(n)}{\tau^{n}} \delta ,
\end{align*}
which, combined with the Schwarz inequality, leads to
\begin{align*}
\int_\R \sum_{j=1}^n 
& \Bigl|\frac{t}{1-t}\mu_j(s)-\frac{|B|^{1/n}}{|A|^{1/n}} \Bigr|\,\rho_A(s)\,ds\\
&\leq  \frac{C(n)}{\tau^{n/2}} \delta^{1/2}\sqrt{\int_\R \max_{i}\mu_i(s)\,\rho_A(s)\,ds}\\
&\leq \frac{C(n)}{\tau^{n/2}} \delta^{1/2}\biggl(\sqrt{\int_\R \max_{i}\Bigl|\frac{t}{1-t}\mu_i(s) -\frac{|B|^{1/n}}{|A|^{1/n}} \Bigr| \,\rho_A(s)\,ds }+\sqrt{\frac{|B|^{1/n}}{|A|^{1/n}}}\biggr)\\
& \leq \frac{C(n)}{\tau^{n/2}} \delta^{1/2}\biggl(\sqrt{\int_\R \sum_{j=1}^n\Bigl|\frac{t}{1-t}\mu_j(s) -\frac{|B|^{1/n}}{|A|^{1/n}} \Bigr| \,\rho_A(s)\,ds }+2\biggr).
\end{align*}
Hence, provided $\delta/\tau^n$ is sufficiently small we get
$$
\int_\R \sum_{j=1}^n\biggl|\frac{t}{1-t}\mu_j(s)-\frac{|B|^{1/n}}{|A|^{1/n}} \biggr|\,ds
\leq \frac{C(n)}{\tau^{n/2}} \delta^{1/2}.
$$
Recalling the definition of $\mu_i$ (see \eqref{eq:mu i}) and using that 
$1-4\delta \leq |B|/|A| \leq 1+4\delta$, we deduce that
\eqref{induction:key} holds.

\subsection{Proof of Lemma \ref{lemma:concavity}}

We first remark that it suffices to prove the result in the case $\hat M=1$, since the general case
follows by applying the result to the function $f/\hat M$.
The proof of this result is rather involved and is divided into several steps.

\subsubsection*{Step a: Making $\psi$ uniformly concave at points that are well separated}
Let $\beta\in(0,1/3]$ to be fixed later, and define $\var :\Omega \to \R$ as 
\begin{equation}
\label{eq:var}
\var(y):=
\left\{
\begin{array}{ll}
\psi(y) +2 - 20\,(\sigma+\varsigma)^{\beta}|y|^2& y \in F,\\
0 &y \in \Omega\setminus F.
\end{array}
\right.
\end{equation}
Notice that, 
\[
|y_{12}'|^2 + |y_{12}''|^2 - |y_1|^2 - |y_2|^2 = -2t'(1-t') |y_1-y_2|^2 \le 
-\frac\tau{2} |y_1-y_2|^2.
\]
Because of this, \eqref{eq:abM} and \eqref{eq:ab}, we have
$0 \leq \var \leq 3$ and
$$
\var(y_1)+\var(y_2)\leq \var(y_{12}')+\var(y_{12}'')+\sigma
- 10\tau \,(\sigma+\varsigma)^{\beta}|y_1-y_2|^2 \qquad 
\forall\,y_1,y_2,y_{12}',y_{12}'' 
\in F,
$$
which implies in particular that
\begin{equation}
\label{eq:barvar}
\var(y_1)+\var(y_2) \leq \var(y_{12}')+\var(y_{12}'')
+\sigma \qquad \forall\,y_1,y_2,y_{12}',y_{12}'' \in F.
\end{equation}
Also, since $\beta \leq 1/3$,
\begin{equation}
\label{eq:barvar2}
\var(y_1)+\var(y_2) < \var(y_{12}')+\var(y_{12}'')
-\tau (\sigma+\varsigma)^{\beta}|y_1-y_2|^2 \qquad 
\forall\,y_1,y_2,y_{12}',y_{12}'' \in F,\,\,
|y_1-y_2| \geq \frac{(\sigma+\varsigma)^{\beta}}{\sqrt\tau},
\end{equation}
that is $\var$ is  uniformly concave on points of $F$ that are at least $(\sigma+\varsigma)^{\beta}/\sqrt\tau$-apart.

\subsubsection*{Step b: Constructing a concave function that should be close to $\var$}
Let us take $\gamma \in (0,\beta]$ to be fixed later,
and define
$$
\bar \var (y):=\min\{\var(y),h\},
$$
where $h \in [0,3]$ is given by 
\begin{equation}
\label{eq:def h}
h:=\inf\bigl\{t>0 : \H^{n-1}(\{\var >t\})
\leq (\sigma+\varsigma)^{\gamma}\bigr\}.
\end{equation}
Since $0\leq \var \leq 3$, we get
\begin{equation}
\label{eq:var bar}
\int_{\Omega} [\var(y) - \bar \var(y)]\,dy =\int_h^{3M} \H^{n-1}(\{\var>s\})\,ds \leq  3M (\sigma+\varsigma)^{\gamma}. 
\end{equation}
Notice that 
whenever $ \max\{\var(y_{12}'),\var(y_{12}'')\} \leq h$,
$\bar \var$ satisfies \eqref{eq:barvar} and \eqref{eq:barvar2}.

We define $\Phi:\Omega\to [0,h]$ to be the concave envelope of 
$\bar\var$, that is, the infimum
among all linear functions that are above $\bar\var$ in $\Omega$.
Our goal is to show that $\Phi$ is $L^1$-close to $\bar\var$ (and hence to $\var$).

\subsubsection*{Step c: The set $\{\Phi=\bar\var\}$ is 
$K(\sigma+\varsigma)^\beta$ 
dense in $\Omega \setminus \co(\{\bar\var > h-K(\sigma+\varsigma)^\beta\})$.}

Let $\beta \in (0,1/3]$ be as in Step b. We claim that there exists a universal constant $K>0$
such that the following holds, provided $\beta$ is sufficiently small (chosen
later depending on $\tau$ and dimension):  For any $y \in \Omega$,\\
-- either there is $x \in \{\Phi=\bar\var\}\cap \Omega$ with $|y-x|\leq K(\sigma+\varsigma)^\beta$;\\
-- or $y$ belongs to 
the convex hull of the set $\{\bar\var > h-K(\sigma+\varsigma)^\beta\}$.\\

To prove this, we define
$$
\Omega_\beta:=\bigl\{y \in \Omega:\dist(y,\partial \Omega) \geq (\sigma+\varsigma)^\beta\bigr\}.
$$
Of course, with a suitable value of $K$, it suffices to consider the case when $y\in \Omega_\beta$.   So, let us fix $y \in \Omega_\beta$.

Since $\Omega$ is a convex set comparable to a ball
of unit size (see \eqref{eq:Ks ball 2}) and $\Phi$ is 
a nonnegative concave function bounded by $3$ inside $\Omega$,
there exists a dimensional constant $C'$ such that, for every linear 
function $L\ge \Phi$ satisfying $L(y) = \Phi(y)$, we have
\begin{equation}
\label{eq:bound gradient f}
|\nabla L| \leq \frac{C'}{(\sigma+\varsigma)^\beta} .
\end{equation}
By \cite[Step 4-c]{fjA}, there are $m\leq n$ points $y_1, \ldots, y_{m} \in F$
such that $y \in S:=\co(\{y_1, \ldots, y_{m}\})$, and
all $y_j$'s are contact points:
\[
\Phi(y_j) =L(y_j)= \bar \var(y_j),  \qquad j = 1, \ldots, m.
\]
If the diameter of $S$ is less than $K(\sigma+\varsigma)^\beta$, then its
vertices are contact points
within $K(\sigma+\varsigma)^\beta$ of $y$ and we are done.

Hence, let us assume that the diameter
of $S$ is at least $K(\sigma+\varsigma)^\beta$.
We claim that 
\begin{equation}\label{eq:x_i}
\bar\var(y_i) > h-K\sigma^\b\qquad \forall\, i =1,\ldots,m.
\end{equation}
Observe that, if we can prove \eqref{eq:x_i}, then
$$
y\in S\subset 
\co(\{\bar\var > h-K(\sigma+\varsigma)^\beta\}),
$$
and we are done again.\\

It remains only to prove \eqref{eq:x_i}.   To begin the proof, given 
$i \in \{1,\ldots,m\}$, take
$j \in \{1,\ldots,m\}$
such that $|y_i-y_j| \ge K(\sigma+\varsigma)^\beta/2$ (such a $j$ always exists because of the assumption on the diameter of $S$).  We rename $i=1$ and
$j=2$.  

Fix $N \in \N$ to be chosen later.  For $x,y\in \Omega$, define
$$
H_N(x,y) :=\bigcap_{k=0}^N \bigcap_{j=0}^k \biggl(\frac{1}{(t')^j(t'')^{k-j}}F[y] - \Bigl(\frac{1}{(t')^j(t'')^{k-j}}-1\Bigr)x \biggr),
$$
where 
$$
F[y]:=F\cap \frac{F-t'y}{1-t'} \cap \frac{F-t''y}{1-t''}.
$$
Observe that, since $\Omega$ is convex and by \eqref{eq:coF},
$$
\H^{n-1}\bigl(\Omega \setminus F[y] \bigr) \leq C
\H^{n-1}(\Omega \setminus F)\leq C\,\varsigma,
$$
\begin{align*}
&\H^{n-1}\biggl(\Om \setminus \biggl(\frac{1}{(t')^j(t'')^{k-j}}F[y] - \Bigl(\frac{1}{(t')^j(t'')^{k-j}}-1\Bigr)x  \biggr)\biggr)\\
&= \frac{1}{(t')^{j(n-1)}(t'')^{(k-j)(n-1)}} \H^{n-1}\Bigl(\bigl((t')^j(t'')^{k-j}\Om + \bigl(1-(t')^j(t'')^{k-j}\bigr)x\bigr)\setminus F[y]\Bigr) \\
&\leq 
\frac{1}{(t')^{j(n-1)}(t'')^{(k-j)(n-1)}} \H^{n-1}(\Om\setminus F[y]) \leq \frac{C}{(t')^{j(n-1)}(t'')^{(k-j)(n-1)}}\,\varsigma,
\end{align*}
(and analogously for $t''$),
so
\begin{equation}
\label{eq:HN}
\begin{split}
\H^{n-1} \bigl(\Om \setminus H_N(x,y)\bigr) &\le C\sum_{k=0}^N\sum_{j=0}^k\frac{1}{(t')^{j(n-1)}(t'')^{(k-j)(n-1)}} \,\varsigma\\
&\leq C \biggl(\frac{1}{t''}\biggr)^{N(n-1)}\, \varsigma,
\end{split}
\end{equation}
where we used that $t'' \leq t'$.
Choose $N$ such that
\begin{equation}
\label{eq:N}
\biggl(\frac{\tau}{2}\biggr)^{N(n-1)}=  C\,\varsigma
\end{equation}
for some large dimensional constant $C$.  
In this way, from \eqref{eq:Ks ball 2} and \eqref{eq:HN} we get
\[
\H^{n-1}\bigl(H_N(y_2,y_1)\bigr) \ge c_n/2> 0, 
\]
and hence $H_N(y_2,y_1)$ is nonempty.

Now, choose $w_0 \in H_N(y_2,y_1)$,
and apply  \eqref{eq:barvar} iteratively in the following way:
if we set  $w_{1}' :=t' w_0+(1-t')y_2$, $w_1'' :=t'' w_0+(1-t'')y_2$, then the fact that $w_0 \in H_N(y_2,y_1)$
implies that $w_1',w_1'' \in F$ (and also that $t'y_1+(1-t')w_0, t''y_1+(1-t'')w_0\in F$). Hence we can apply \eqref{eq:barvar} to obtain
$$
\var(w_1')+ \var(w_1'')\geq  \var(y_2)+ \var(w_0)-\sigma.
$$
Then define $w_1$ to be equal either to $w_1'$ or to $w_1''$ so that 
$\var(w_1)=\max\{\var(w_1'),\var(w_1'')\}$. Then, it follows from the  equation above that
$$
\var(w_1) \geq \frac{ \var(y_2)+ \var(w_0)}{2} - \frac{\sigma}2.
$$
We now set $w_{2}' :=t' w_1+(1-t')y_2,w_2'' :=t'' w_1+(1-t'')y_2 \in F$,
and apply \eqref{eq:barvar} again to get
$$
 \var(w_2')+ \var(w_2'')\geq  \var(y_2)+ \var(w_1)-\sigma\geq (1-1/4)\var(y_2)
 +\var(w_0)/4-\bigl(1+1/2\bigr)\sigma.
$$
Again we choose $w_2 \in \{w_2',w_2''\}$ so that $\var(w_2)=\max\{\var(w_2'),\var(w_2'')\}$
and we keep iterating this construction, so that in $N$ steps we get
(recall that $0\leq  \var\leq 3$)
\begin{equation*}
\begin{split}
\var(w_N)&\geq (1-2^{-N})\var(y_2)+2^{-N}\var(w_0)- 2\sigma\\
&\geq (1-2^{-N})\var(y_2) -2\sigma\\
&\geq \var(y_2)-3 \cdot 2^{-N} - 2\sigma
\end{split}
\end{equation*}
Hence,
\begin{equation}\label{eq:varz1}
\bar\var(w_N) \geq \bar \var(y_2)-3 \cdot 2^{-N} - 2\sigma.
\end{equation}
In addition, since $w_0 \in H_N(y_1,y_2)$, 
\[
y' := t'y_1+ t'' w_N \in F, \qquad y'' := t''y_1+ t'w_N\in F.
\]
Since  the diameter of $F$ is bounded (see \eqref{eq:coF} and
\eqref{eq:Ks ball 2}) it is easy to check that
$$
|w_N- y_2|\leq C\, (t')^{N} \le C(1-\tau/2)^N.
$$
Therefore, by \eqref{eq:bound gradient f} we have
\[
|L(y'+y'') - L(y_1+y_2)| = |L(w_N-y_2)| \le C\,(\sigma+\varsigma)^{-\beta}(1-\tau/2)^{N}.
\]
Hence, since $y_1$ and $y_2$ are contact points and $L \geq \bar\var$,
using \eqref{eq:varz1} and \eqref{eq:N} we get
\begin{equation}
\label{eq:mid}
\begin{split}
\bar\var(y_1) + \bar\var(w_N) 
& \ge \bar\var(y_1) + \bar\var(y_2) - 3 \cdot 2^{-N}- 2\sigma \\
&= L(y_1+y_2)- 3 \cdot 2^{-N}- 2\sigma \\
&\geq L(y'+y'') -C\Bigl(2^{-N}+\sigma+(\sigma+\varsigma)^{-\beta}(1-\tau/2)^{N}\Bigr)\\
&\geq \bar\var(y')+\bar\var(y'') -C\,\bigl(\sigma+\varsigma\bigr)^{\min \{\theta,\kappa-\beta\}},
\end{split}
\end{equation}
where (recall \eqref{eq:tau12})
\begin{equation}
 \label{eq:theta kappa}
 \theta:=\frac{1}{n-1} \frac{\log(2)}{|\log(\tau/2)|},\qquad \kappa:=\frac{1}{n-1} \frac{|\log(1-\tau/2)|}{|\log(\tau/2)|}.
\end{equation}
Now assume by way of contradiction that 
\[
\bar\var(y_1) \le h - K (\sigma + \varsigma)^\beta.
\]
We also have (recalling that $t'' = 1-t'$) 
\begin{align*}
L(w_N) & \le L(w_N-y_2) + L(y_2) = L(w_N-y_2) + \bar \var(y_2)  \\
& \le C(\sigma + \varsigma)^{-\beta}(1-\tau/2)^N + h 
\le h  + C(\sigma +\varsigma)^{\kappa - \beta}
\end{align*}
Hence, since $L(y_1) = \bar\var(y_1)$, 
\begin{align*}
L(y'') & = t''\bar\var(y_1) + t' L(w_N) \\
& \le t''h -  \frac{K}2 (\sigma + \varsigma)^\beta
+ t' h + C(\sigma +\varsigma)^{\kappa - \beta}  < h, 
\end{align*}
provided $K > 2C$ and $ \beta \le \kappa/2$.   Similarly
(and more easily since $t'' \le t'$), we have
$L(y')< h$.  Since $L \ge \bar\var$, we have 
$\max\{\bar\var(y'),\bar\var(y'')\}<h$.  Applying
\eqref{eq:barvar2} with $y_2$ replaced by $w_N$ we get
$$
\bar\var(y')+\bar\var(y'')  = 
\var(y')+\var(y'')  
\ge \var(y_1) + \var(w_N) + \tau(\sigma+\varsigma)^\beta|y_1-w_N|^2,
$$
and since $|y_1-w_N| \geq |y_1-y_2|/2 \geq K (\sigma+\varsigma)^\beta/4$ this implies
$$
\bar\var(y')+\bar\var(y'') 
\ge \bar\var(y_1) + \bar\var(w_N) + \frac{\tau K^2}{16}(\sigma+\varsigma)^{3\beta},
$$
which contradicts \eqref{eq:mid} provided we choose
$\beta:=\min\left\{\frac{\theta}{3},\frac{\kappa}{4}\right\}$ and $K$ sufficiently large.

Recalling the definition of $\theta$ and $\kappa$ (see \eqref{eq:theta kappa}), this concludes the proof with the choice
\begin{equation}
\label{eq:beta}
\beta:=\frac{1}{(n-1)|\log(\tau/2)|}\min\left\{\frac{\log(2)}{3},\frac{|\log(1- \tau/2)|}{4}\right\} \ge \frac{\tau}{8(n-1)|\log \tau|}.
\end{equation}

\subsubsection*{Step d: Most of the level sets of $\bar\var+20(\sigma+\varsigma)^\beta|y|^2$ are close to their convex hull}

This will follow from the fact that it is true for $\psi$.  Indeed,
define
\[
\bar\psi(y):=\min\left\{\bar\var(y)+20(\sigma+\varsigma)^\beta|y|^2,h\right\}
\]
Then
\[
\{y\in F: \bar\psi(y) > s\} = 
\begin{cases}
\{y\in F: \psi(y) > s-2\} & \quad \text{if } s< h, \\
\qquad \emptyset  & \quad \text{if } s\ge h.
\end{cases}
\]
Define
\[
H_1:=\{s \in \R\,:\,s-2 \in H\} \cup [h,\infty),
\qquad H_2:=\R\setminus H_1.
\]
Then it follows from \eqref{eq:level sets co} and \eqref{eq:coF} that
\begin{equation}
\label{eq:co psi}
\int_{H_1} \H^{n-1}\bigl(\co(\{\bar \psi>s\})\setminus \{\bar \psi>s\}\bigr)\,ds
+\int_{H_2} \H^{n-1}\bigl(\{\bar \psi>s\}\bigr)\,ds 
\le 3\, \H^{n-1}(\Omega\setminus F) + \varsigma \leq C\,\varsigma.
\end{equation}
Notice that, since $\{\bar\psi>s\}\supset \{\bar\var>s\}$, by \eqref{eq:def h} we 
have $\H^{n-1}\bigl(\{\bar\psi>s\}\bigr)
\geq (\sigma+\varsigma)^\gamma$ for all $s<h$.
So by \eqref{eq:co psi}
\begin{equation}
\label{eq:co psi2}
\H^1(H_2) \leq C\, (\sigma+\varsigma)^{1-\gamma}.
\end{equation}

\subsubsection*{Step e: $\psi$ is $L^1$-close to a concave function}
Since the sets $\{\bar\psi>s\}$ are decreasing in $s$, so are their convex hulls
$\co(\{\bar\psi >s\})$. Hence, we can
define a new function $\xi:\Omega \to \R$ with convex level sets given by 
\[
\{\xi>s\}:=\co(\{\bar\psi >s\})\quad \text{if $s \in H_1$},
\qquad \{\xi>s\}:=\bigcap_{s' \in H_1,\,s'<s}\co(\{\bar\psi>s' \}) \quad \text{if $s \in H_2$},
\]
Recall that $\Phi$ denotes the concave envelope of $\bar\phi$, and 
in particular, $\bar\psi \le \Phi + C(\sigma + \varsigma)^\beta$.  It follows 
from the definition of $\xi$ that
\begin{equation}
\label{eq:xi extra}
0 \le \bar\psi \le \xi; \quad \xi \le \Phi
+C(\sigma+\varsigma)^\beta.
\end{equation}
By \eqref{eq:co psi} and \eqref{eq:co psi2}, we see that
$\xi$ satisfies
\begin{equation}
\label{eq:xi}
\int_{\Omega} |\xi-\bar\psi| \leq \varsigma+C(\sigma+\varsigma)^{1-\gamma}.
\end{equation}
Also, because of \eqref{eq:def h}, we see that
\begin{equation}
\label{eq:level xi}
\H^{n-1}(\{\xi>s\}) \geq (\sigma+\varsigma)^\gamma \qquad \forall\,0\leq s<h.
\end{equation}
Recall from Step c that the contact set 
$\{\Phi = \bar\var\}$ is $\e_1$-dense
in 
\[
\Omega \setminus \co(\{\bar\var > h-\e_1 \})
\]
with $\e_1 := K(\sigma+\varsigma)^\beta$.  \\

Let 
\begin{equation} \label{eq:epsilon}
\e := \hat C\frac{\e_1}{(\sigma+ \varsigma)^{ \gamma}  }
=  \hat C (\sigma+ \varsigma)^{\beta - \gamma},
\end{equation}
where $\hat C$ is a large universal constant (to be chosen).

We claim that, if $s < h-\e_1$, 
\begin{equation}\label{eq:convex cover}
\{y\in \Omega: \Phi(y)>s\} \subset \mbox{$\e$-neighborhood of} \ 
\{y\in \Omega: \xi(y)>s\}.
\end{equation}
To prove \eqref{eq:convex cover}, assume by contradiction that 
there exists $y\in \{\Phi > s\}$ such that 
$B_\e(y) \cap \{\xi >s\}= \emptyset$.  Since, $s + \e_1 < h$,
\eqref{eq:level xi} implies
\[
\H^{n-1}(\{\xi > s + \e_1\}) \ge (\sigma + \varsigma)^\gamma.
\]
In addition, $\{\xi > s + \e_1\}$ is a universally bounded convex set, 
so there is $y' \in \Omega$ and $\rho = c (\sigma + \varsigma)^\gamma$,
with $c>0$ a dimensional constant, such that 
\[
B_\rho(y') \subset \{ \xi > s + \e_1\} \subset \{\Phi > s\}.
\]
Since $\Omega$ is universally bounded, there exists $y''\in \Omega$
and $r = c\rho \e$, with $c>0$ a dimensional constant, such that
\[
B_r(y'') \subset \co(B_{\rho}(y') \cup \{ y\}) \cap B_{\e}(y) 
\subset \{ \Phi > s\} \cap \{\xi \le s\}.
\]
Thus for any $z\in B_r(y'')$, 
\[
\bar\var(z) \le \xi(z)  \le s < \Phi(z),
\]
and there are no contact points of $\{ \Phi = \bar\var\}$ in $B_r(y'')$. 
But note that for our choice of $\e$, $r= c\rho \e > \e_1$ provided $\hat C$ is sufficiently large.
This contradicts the $\e_1$-density property,
proving \eqref{eq:convex cover}.\\

Since all level sets of $\xi$ are (universally) bounded convex sets, 
as a consequence of \eqref{eq:convex cover} we deduce that 
\[
\H^{n-1}(\{\Phi>s\}) \leq \H^{n-1}(\{\xi>s\})+C \e \qquad \forall\, s < h-\e_1.
\]
Furthermore, since $\xi \leq \Phi + \e_1$, we obviously have
that $|\Phi-\xi|\leq 2\e_1$ on the set $\{\xi>h-\e_1\}$.
Hence, by Fubini's Theorem (and $\e > \e_1$),
$$
\int_{\Omega} |\Phi-\xi| \leq C \,\e.
$$
Combining this estimate with \eqref{eq:xi} and the fact that
$|\bar\psi - \bar\var| \le C(\sigma + \varsigma)^\b  < \e$, we get
\[
\int_{\Omega} |\Phi-\bar\var| \leq C \, \e 
\]
By construction
$|\var(y)-2 - \psi (y)| \leq C\,(\sigma+\varsigma)^\beta$ on $F$ (see \eqref{eq:var}).
Hence, by \eqref{eq:var bar} we have
\begin{align*}
\int_{\Omega}|\bar\var(y)-2 - \psi (y)|\,dy&
\leq 3\,|\Omega\setminus F| + 
\int_{F}\bigl[ |\bar\var(y)-\var(y)| + |\var(y)-2M-f(y)|\bigr]\,dy 
\\
&\leq 3\, \bigl((\sigma+\varsigma)^\gamma+\varsigma\bigr) +  \e|F| \leq C\bigl( (\sigma+\varsigma)^\gamma+\e\bigr).
\end{align*}
Taking $\gamma = \beta/2$ and recalling \eqref{eq:epsilon} and \eqref{eq:beta}, 
this proves \eqref{eq:almost concave} with $\Psi:=\Phi-2$.

\subsection{A linearity result}
The aim of this section it to show that, if a one dimensional function satisfies on a large set the concavity-type estimate in Remark \ref{rmk:4pts} (see Section \ref{sect:natural})
both from above and from below, then it is universally close to a linear function.

\begin{lemma}
\label{lemma:1d}
Let $0 < \tau \le1/2$ and fix $t'$ such that 
$1/2 \le t' \le 1-\tau/2$.  Let $t'' = 1-t'$, and for all 
$m_1,m_2\in \R$ define
\[
m_{12}' := t'm_1 + t'' m_2; \quad m_{12}'' := t'' m_1 + t' m_2.
\]
Let $E\subset \R$, and let $f:E\to \R$ be a bounded measurable function such that 
\begin{equation}
\label{eq:f}
\left|f(m_1)+f(m_2) -f(m_{12}')-f(m_{12}'') \right| \leq 1 \qquad \forall\,m_1,m_2,m_{12}',m_{12}'' \in E.
\end{equation}
Assume that there exist points $\bar m_1,\bar m_2 \in \R$ such that $\bar m_1,\bar m_2,\bar m_{12}',\bar m_{12}'' \in E$
and  $|E\cap [\bar m_1,\bar m_2]| \geq (1-\e)|\bar m_2-\bar m_1|$.
Then the following hold provided $\e$ is sufficiently small (the smallness being universal):
\begin{enumerate}
\item[(i)] There 
exist a linear function $\ell:[\bar m_1,\bar m_2] \to \R$ and 
a universal constant $\bar M$,
such that 
$$
|f-\ell |\leq \bar M \qquad \text{in $E \cap [\bar m_1,\bar m_2]$}.
$$
\item[(ii)] If in addition $|f(\bar m_1)|+|f(\bar m_2)| \leq K$ for some constant $K$, then $|f|\leq K+\bar M$
inside $E$.
\end{enumerate}
\end{lemma}
\begin{proof} Without loss of generality, we can assume that $[\bar m_1,\bar m_2]=[-1,1]$
and $E \subset [-1,1]$.
Given numbers $a\in\R$ and $b>0$, we write 
$a=O(b)$ if $|a|\leq Cb$ for some universal
constant $C$.\\

To prove (i), let us define
$$
\ell(m):=\frac{f(1)-f(-1)}{2} m + \frac{f(1)+f(-1)}{2},
$$
and set $F:=f-\ell$. 
Observe that $F(-1)=F(1)=0$, and $F$ still satisfies \eqref{eq:f}.
Hence, 
since by assumption $-1,1,1-2t',1-2t'' \in E$, by \eqref{eq:f} we get $|F(1-2t')+F(1-2t'')| \leq 1$.
Let us extend $F$ to the whole interval $[-1,1]$ as $F(m)=0$ if $m \not\in E$, and set
$$
\bar M:=\sup_{m \in [-1,1]}|F(m)|.
$$
We want to show that $\bar M$ is universally bounded.

Averaging \eqref{eq:f} (applied to $F$ in place of $f$) with respect to $m_2 \in E$
and using that $|E\cap [-1,1]| \geq 2(1-\e)$, we easily obtain the following bound:
\begin{align*}
F(m_1)&=\frac{1}{2-2t'}\int_{t'm_1-(1-t')}^{t'm_1+(1-t')}F(m)\,dm
+\frac{1}{2-2t''}\int_{t''m_1-(1-t'')}^{t''m_1+(1-t'')}F(m)\,dm\\
&\qquad - \frac12\int_{-1}^{1}F(m)\,dm +O(1)+O(\e \bar M),
\end{align*}
from which it follows that
\begin{equation}
\label{eq:c almost Lip}
|F(m')-F(m'')|\leq  C\Bigl( \bar M|m'-m''|+1+ \e \bar M\Bigr)\qquad \forall\,m',m'' \in E.
\end{equation}

Now, pick a point $\tilde m\in E$ such that
\begin{equation}
\label{eq:F large}
|F(\tilde m)|\geq \bar M-1.
\end{equation}
With no loss of generality we assume that $F(\tilde m) \geq \bar M-1$.
Since
$$
\bigcup_{\bar m_0\in [-1,1]} \{-t''+(1-t'')\bar m_0\} =[-1,1-2t''],\qquad \bigcup_{\bar m_0\in [-1,1]} \{t'\bar m_0+(1-t')\} =[1-2t'',1],
$$
we can find a point $\bar m_0 \in [-1,1]$ such that
$$
\text{either}\quad \tilde m=-t''+(1-t'')\bar m_0\quad \text{or}\quad \tilde m=t'\bar m_0+(1-t').
$$
Without loss of generality we can assume that we are in the fist case.

Set $m_0:=-1$. We want to find a point $\hat m_0 \in E$ close to $\bar m_0$ such that 
\begin{equation}
\label{eq:close bar m0}
t'm_0+(1-t')\hat m_0,\,t''m_0+(1-t'')\hat m_0 \in E.
\end{equation}
Define $C_t:=2\left(\frac{1}{1-t'}+\frac{1}{1-t''}\right)$. Then the above inclusions mean
$$
\hat m_0 \in \frac{E-t'm_0}{1-t'}\cap \frac{E-t''m_0}{1-t''},
$$
and since the latter set contains $[-1,1]$ up to a set of measure $C_t\e$, we can find such a point at a distance at most $C_t\e$ from $\bar m_0$.
Notice that in this way we also get 
$|t''m_0+(1-t'')\hat m_0 - \tilde m| \leq C_t\e$, so, by \eqref{eq:c almost Lip} and \eqref{eq:F large},
$$
F(t''m_0+(1-t'')\hat m_0) \geq \bar M-1 -  C\bigl(1+(1+C_t)\e \bar M \bigr).
$$
Then, thanks to \eqref{eq:close bar m0}, we can apply \eqref{eq:f} with $F$ in place of $f$, $m_1=m_0$, and $m_2=\bar m_0$, to deduce that (recall that $F(m_0)=F(-1)=0$ and that $|F|\leq \bar M$)
\begin{align*}
F(t'm_0+(1-t')\hat m_0) & \leq 1+F(m_0)-F(t''m_0+(1-t'')\hat m_0)+ F(\bar m_0)\\
&\leq 1+(1-\bar M)+ C\bigl(1+(1+C_t)\e \bar M \bigr) +\bar M\\
&=2+  C\bigl(1+(1+C_t)\e \bar M\bigr).
\end{align*}
We now define $m_1:=t'm_0+(1-t')\hat m_0$ and we choose $\bar m_1\in [-1,1]$ such that 
$$
\tilde m=t''m_1+(1-t'')\bar m_1.
$$
Again we pick a point $\hat m_1 \in [\bar m_1-C_t\e,\bar m_1+C_t\e]\cap E$ such that
$$
tm_1'+(1-t')\hat m_1,\,t''m_1+(1-t'')\hat m_1 \in E,
$$
and applying again \eqref{eq:c almost Lip} and \eqref{eq:f} we get
$$
F(t''m_1+(1-t'')\hat m_1) \geq \bar M-1 - \bar C\bigl(1+(1+C_t)\e \bar M \bigr),
$$
hence
\begin{align*}
F(t'm_1+(1-t')\hat m_1)&
\leq 1+F(m_1)-F(t''m_1+(1-t'')\hat m_1)+ F(\bar m_1)\\
& \leq 1+2+ \bar C\bigl(1+(1+C_t)\e \bar M\bigr)+(1-\bar M)+ \bar C\bigl(1+(1+C_t)\e \bar M\bigr)+\bar M\\
&\leq 4+2 C\bigl(1+(1+C_t)\e \bar M\bigr).
\end{align*}
Iterating this procedure, after $k$ steps we get 
$$
F(t'm_k+(1-t')\hat m_k) \leq 2(k+1)+(k+1)\bar C\bigl(1+(1+C_t)\e \bar M\bigr),
$$
and it is easy to check that the points $m_k$ and $\bar m_k$ converge geometrically to $\tilde m$ up to an additive error $C_t\e$ at every step, that is
$$
|m_k -\tilde m|+|\bar m_k - \tilde m| \leq \bar C \Bigl( 2^{-k}+kC_t\e\Bigr).
$$
Hence, thanks to \eqref{eq:c almost Lip} applied with $m'=t'm_k+(1-t')\hat m_k$ and $m''=\tilde m$ we get
\begin{align*}
\bar M-1-2(k+1)\Bigl(2+ \bar C\bigl(1+(1+C_t)\e \bar M\bigr)\Bigr) &\leq F(\tilde m)-F(t'm_k+(1-t')\bar m_k)\\
&\leq  \bar C\Bigl( \bar M \bigl(2^{-k}+kC_t\e\bigr)+ \e \bar M\Bigr),
\end{align*}
for some universal constant $\bar C$.
Hence, by choosing $k=N$  a large universal constant so that $\bar C\,2^{-N}\leq 1/2$, we obtain
$$
\bar M \leq \frac{\bar M}{2} +C\biggl(1+N+\e N \bar M \biggr),
$$
which proves that $\bar M$ is universally bounded
provided $\e$ is sufficiently small (the smallness being universal).
This proves (i).\\

To prove (ii), it suffices to observe that
if $|f(-1)|+|f(1)| \leq K$ then $|\ell| \leq K$, so (i) gives
$$
|f| \leq |\ell|+|F| \leq K+\bar M.
$$
\end{proof}

\subsection{Proof of Proposition \ref{prop:coS}}
After  translating and replacing $R$ by $2R$, we can assume that
the barycenter of both $K_A$ and $K_B$ coincide with the origin.  Then
observe that 
\begin{align*}
\|\chi_{tA}\ast \chi_{(1-t)B} - \chi_{tK_A}\ast \chi_{(1-t)K_B}\|_{L^\infty(\R^n)}&\leq 
\|\chi_{tA}- \chi_{tK_A}\|_{L^1(\R^n)}+ \|\chi_{(1-t)B} - \chi_{(1-t)K_B}\|_{L^1(\R^n)}\\
& \leq |A\Delta K_A|+|B\Delta K_B|\leq \zeta.
\end{align*}\\

We claim that 
\begin{equation}
\label{eq:chi KAKB}
\chi_{tK_A}\ast \chi_{(1-t)K_B}(x) 
> \zeta \qquad \forall\,x \in \bigl(1-C\zeta^{1/n}\bigr) [tK_A+(1-t)K_B],
\end{equation}
for $C$ a universal constant.  Indeed, by John's lemma, since 
$K_A$ and $K_B$ are convex sets in $B_R$ of volume comparable 
to $1$ with barycenter $0$, there is a ball $B_c$ centered at $0$ 
such that 
\[
B_c \subset K_A\cap K_B, \qquad c \ge c_n R^{1-n}.
\]
Since we are assuming $R\le \tau^{-N_n}$, $c$ is a universal positive constant.  
Let $x= tx_1 + (1-t)x_2$ for some $x_1\in (1-\d_1)K_A$ and $x_2 \in (1-\d_1)K_B$,
then 
\[
\d_1 B_c + x_1 \subset K_A, \quad \d_1B_c + x_2 \subset K_B.
\]
Hence $\d_1B_c \subset (x_1-K_A)$ and $\d_1B_c \subset (K_B - x_2)$, and
consequently
\[
\tau \d_1 B_c \subset [t(x_1-K_A)] \cap [(1-t)(K_B - x_2)].
\]
Thus
\[
\chi_{tK_A}\ast \chi_{(1-t)K_B}(x) = |(x-tK_A)\cap (1-t)K_B|
= |t(x_1-K_A) \cap (1-t)(K_B-x_2)| \ge |\tau \d_1 B_c| > \zeta
\]
provided $\d_1 = C\zeta^{1/n}$ for some universal constant $C$, proving the claim.\\

It follows from \eqref{eq:chi KAKB} that $\chi_{tA}\ast \chi_{(1-t)B}(x)>0$,
which implies $x \in S$.  In all, we have 
\begin{equation}
\label{eq:KABS}
\bigl(1-C\zeta^{1/n}\bigr) [tK_A+(1-t)K_B]\subset S.
\end{equation}
Therefore by \eqref{eq:measures} (since, by assumption, $\delta \leq \zeta$)
$$
|tK_A+(1-t)K_B| \leq (1+C\zeta^{1/n})|S|\leq 1+C\,\zeta^{1/n}.
$$
Since
\begin{equation}
\label{eq:measKAB}
\bigl| |K_A|-1\bigr| +\bigl| |K_B|-1\bigr|  \leq C\,\zeta
\end{equation}
(by \eqref{eq:measures} and \eqref{eq:KA KB zeta}),
it follows from Theorem \ref{thm:BMconvex} that
\begin{equation}
\label{eq:coAB}
|K_A\Delta K_B|\leq C\,
\zeta^{1/2n}.
\end{equation}
(Notice that since $K_A$ and $K_B$ have the same barycenter, there is
not need to translate them.)  In particular this immediately implies
that
$$
|A\Delta B|\leq C\,
\zeta^{1/2n}.
$$

Now observe that, by \eqref{eq:BM} and \eqref{eq:measKAB} we get
$$
\bigl|\bigl(1-C\zeta^{1/n}\bigr) [tK_A+(1-t)K_B]\bigr| \geq 1-C\,\zeta^{1/n},
$$
and hence it follows from \eqref{eq:KABS} and \eqref{eq:measures} that
\begin{equation}
\label{eq:coABS}
\bigl|\bigl(tK_A+(1-t)K_B\bigr)\Delta S\bigr| \leq C\,\zeta^{1/n}.
\end{equation}
Consider  the convex set $\K_0:=\co(K_A\cup K_B)\supset tK_A+(1-t)K_B$.
By a simple geometric argument using \eqref{eq:coAB}, 
we easily deduce that
$$
\K_0\subset \bigl(1+C\,\zeta^{1/2n^2}\bigr)K_A,
\quad \K_0\subset \bigl(1+C\,\zeta^{1/2n^2}\bigr)K_B,
\quad \K_0\subset \bigl(1+C\,\zeta^{1/2n^2}\bigr)[tK_A+(1-t)K_B],
$$
so, by \eqref{eq:KA KB zeta} and \eqref{eq:coABS} we obtain
\begin{equation}
\label{eq:ABK0}
|A\Delta \K_0|+|B\Delta \K_0| +|S\Delta \K_0| \leq C\,\zeta^{1/2n^2}.
\end{equation}\\

Finally, we claim that 
$$
A\subset (1+C\,\zeta^{1/2n^3})\K_0.
$$
Indeed, following the argument used in the proof of \cite[Lemma 13.3]{christ1}
(see also \cite[Proof of Theorem 1.2, Step 5]{fjA}),
let $x \in A\setminus \K_0$, denote by $x'\in \partial\K_0$ the closest point
in $\K_0$ to $x$, set $\rho:=|x-x'|=\dist(x,\K_0)$, and let $v \in \mathbb S^{n-1}$ be the unit normal to a supporting
hyperplane to $\K_0$ at $x'$, that is
$$
(z-x')\cdot v \leq 0 \qquad \forall\,z \in \K.
$$
Let us define $\K_\rho:=\left\{z \in \K_0:(z-x')\cdot v \geq -\frac{t}{1-t}\rho\right\}$.
Observe that, since $\K_0$ is a bounded convex set with volume 
close to $1$, $|\K_\rho| \geq c_n\tau^n \rho^n$ for some dimensional constant $c_n>0$.
Since $x \in A$ we have
$$
S=tA+(1-t)B \supset \bigl(tx+(1-t)[\K_\rho\cap B]\bigr) \cup (S\cap \K_0),
$$
and the two sets in the right hand side are disjoint.
This implies that (see \eqref{eq:ABK0})
$$
|S| \geq \tau^n \bigl(|\K_\rho| - |\K_0 \setminus S|\bigr) +|S\cap \K_0|
\geq \rho^n/C +|S| - C\,\zeta^{1/2n^2},
$$
from which we deduce 
$$
\rho\leq C\,\zeta^{1/2n^3}.
$$
Since $x$ is arbitrary, this implies that $A$ is contained inside the
$\left(C\,\zeta^{1/2n^3}\right)$-neighborhood of $\K_0$, proving the claim.\\

Since the analogous statement holds for $B$, we obtain that
$$
A\cup B \subset \K:=(1+C\,\zeta^{1/2n^3})\K_0,
$$
and (thanks to \eqref{eq:ABK0})
$$
|\K_0\setminus A|+|\K_0\setminus B| \leq C\,\zeta^{1/2n^3},
$$
as desired.



\begin{thebibliography}{99}


\bibitem[AGS]{AGS}
Ambrosio, L.; Gigli, N.; Savar\'e, G.
\textit{Gradient flows in metric spaces and in the space of probability measures.}
Second edition. Lectures in Mathematics ETH Z\"urich. Birkh\"auser Verlag, Basel, 2008.

\bibitem[BB1]{BB1} Ball, K. M.; B\"or\"oczky, K. J. Stability of some versions of the Pr\'ekopa-Leindler inequality.
{\it Monatsh. Math.} 163 (2011), no. 1, 1-14. 

\bibitem[BB2]{BB2}
Ball, K. M.; B\"or\"oczky, K. J. 
Stability of the Pr\'ekopa-Leindler inequality. {\it Mathematika} 56 (2010), no. 2, 339-356.

\bibitem[BP]{BP}
Brasco, L.; Pratelli, A. Sharp stability of some spectral inequalities.
{\it Geom. Funct. Anal.} 22 (2012), no. 1, 107-135.

\bibitem[C1]{christ1}
Christ M.
Near equality in the two-dimensional Brunn-Minkowski inequality.
Preprint, 2012. Available online at {\tt http://arxiv.org/abs/1206.1965}

\bibitem[C2]{christ2}
Christ M.
Near equality in the Brunn-Minkowski inequality.
Preprint, 2012. Available online at {\tt http://arxiv.org/abs/1207.5062}

\bibitem[C3]{christ3}
Christ M.
An approximate inverse Riesz-Sobolev inequality.
Preprint, 2012. Available online at {\tt http://arxiv.org/abs/1112.3715}


\bibitem[C4]{christ4}
Christ M.
Personal communication.

\bibitem[D]{Disk}
Diskant, V. I.
Stability of the solution of a Minkowski equation. (Russian)
{\it Sibirsk. Mat. \v Z.} 14 (1973), 669-673, 696. 

\bibitem[FJ]{fjA}
Figalli A.; Jerison D.
Quantitative stability for sumsets in $\R^n$.
\textit{J. Europ. Math. Soc. (JEMS)}, to appear.



\bibitem[FMP1]{fmpK}
Figalli, A.; Maggi, F.; Pratelli, A.
A mass transportation approach to quantitative isoperimetric inequalities.
\textit{Invent. Math.} 182 (2010), no. 1, 167-211.

\bibitem[FMP2]{fmpBM}
Figalli, A.; Maggi, F.; Pratelli, A.
A refined Brunn-Minkowski inequality for convex sets.
\textit{Ann. Inst. H. Poincar\'e Anal. Non Lin\'eaire} 26 (2009), no. 6, 2511-2519.

\bibitem[G]{Groe}
Groemer, H.
On the Brunn-Minkowski theorem.
{\it Geom. Dedicata} 27 (1988), no. 3, 357-371. 


\bibitem[J]{john}
John F.
{\em Extremum problems with inequalities as subsidiary conditions.}
In Studies and Essays Presented to R. Courant on his 60th
  Birthday, January 8, 1948, pages 187-204. Interscience, New York, 1948.

\bibitem[S]{Sch}
Schneider, R. \textit{Convex bodies: the Brunn-Minkowski theory.} Encyclopedia of Mathematics and its Applications, 44. Cambridge University Press, Cambridge, 1993.


\end{thebibliography}
\end{document}